\pdfoutput=1
\documentclass[11pt,reqno]{amsart}

\usepackage[letterpaper,hmargin=1.25in,vmargin=1.22in]{geometry}
\usepackage{amssymb}
\usepackage[all]{xy}
\SilentMatrices
\newcommand{\cxymatrix}[1]{\vcenter{\xymatrix{#1}}}
\usepackage[hyphens]{url}

\newcommand{\dfn}[1]{\textbf{\boldmath{#1}}}

\newtheorem{de}{Definition}[section]
\newtheorem{lem}[de]{Lemma}
\newtheorem{prop}[de]{Proposition}
\newtheorem{cor}[de]{Corollary}
\newtheorem{thm}[de]{Theorem}

\theoremstyle{remark}
\newtheorem{rem}[de]{Remark}
\newtheorem{ex}[de]{Example}

\DeclareMathOperator*{\colim}{colim}
\DeclareMathOperator{\supp}{supp}
\DeclareMathOperator{\im}{Im}
\DeclareMathOperator{\Prin}{Prin}
\DeclareMathOperator{\Bun}{Bun}
\DeclareMathOperator{\VB}{VB}
\DeclareMathOperator{\od}{od}
\DeclareMathOperator{\ev}{ev}
\DeclareMathOperator{\assoc}{assoc}
\DeclareMathOperator{\frme}{frame}
\DeclareMathOperator{\GL}{GL}
\DeclareMathOperator{\Isom}{Isom}

\newcommand{\Diff}{{\mathfrak{D}\mathrm{iff}}}
\newcommand{\DVect}{{\mathfrak{D}\mathrm{Vect}}}

\def \N{\mathbb{N}}
\def \Z{\mathbb{Z}}
\def \Q{\mathbb{Q}}
\def \C{\mathbb{C}}
\def \R{\mathbb{R}}

\usepackage{ifpdf}
\ifpdf  \usepackage[pdftex,bookmarks=false]{hyperref}
\else   \usepackage[dvips,bookmarks=false]{hyperref}
        \usepackage{breakurl} %
\fi
\usepackage{color}
\definecolor{darkgreen}{rgb}{0,0.45,0}
\hypersetup{colorlinks,urlcolor=blue,citecolor=darkgreen,linkcolor=darkgreen,linktocpage}

\title{Smooth classifying spaces}

\author{J. Daniel Christensen}
\email{jdc@uwo.ca}
\address{Department of Mathematics, University of Western Ontario, London, Ontario, Canada}

\author{Enxin Wu}
\email{exwu@stu.edu.cn}
\address{Department of Mathematics, Shantou University, Guangdong, P.R. China}

\date{December 27, 2020}

\begin{document}

\begin{abstract}
We develop the theory of smooth principal bundles for a smooth group $G$,
using the framework of diffeological spaces.
After giving new examples showing why arbitrary principal bundles cannot
be classified, we define $D$-numerable bundles, the smooth analogs of
numerable bundles from topology, and prove that pulling back a $D$-numerable bundle
along smoothly homotopic maps gives isomorphic pullbacks.
We then define smooth structures on Milnor's spaces $EG$ and $BG$,
show that $EG \to BG$ is a $D$-numerable principal bundle,
and prove that it classifies all $D$-numerable principal bundles
over any diffeological space.
We deduce analogous classification results for $D$-numerable diffeological
bundles and vector bundles.
\end{abstract}



\maketitle

\tableofcontents

\section{Introduction}

The theory of classifying spaces for principal bundles has a long history
in topology~\cite{Mi,Se,St}
and its importance is well-established.
In this paper, we develop the analogous theory for spaces with a smooth structure.
In brief, given a smooth group $G$, we put a smooth structure on $EG$ and $BG$,
define a smooth principal $G$-bundle $EG \to BG$, and show that this bundle is
universal in an appropriate sense.
We show how these results can be used to classify smooth fiber bundles as
well as smooth vector bundles, laying the foundation for future work on
smooth characteristic classes and smooth $K$-theory.

The framework we use to formulate these results is that of diffeological spaces.
A diffeological space is a set $X$ along with a chosen collection of functions
$U \to X$ (called plots), where $U$ runs over open subsets of Euclidean spaces.
The plots are subject to three simple axioms (see Definition~\ref{de:diffeological-space}).
The category of diffeological spaces and smooth maps between them is a
convenient category in which to make constructions.  It includes smooth manifolds
as a full subcategory, as well as function spaces, diffeomorphism groups and
singular spaces such as manifolds with corners and all quotients.
The geometry and homotopy theory of diffeological spaces is well-developed
(see~\cite{CSW,CW1,CW2,I1,I2,So1,So2,Wu} and references therein),
giving a solid framework in which to develop the present theory.

When classifying principal bundles in topology, one has to either 
restrict to base spaces that are paracompact or (more generally) 
consider only numerable bundles.
The issue is that it is not true in general that if $\pi : E' \to B'$
is a principal bundle and $f, g : B \to B'$ are homotopic, then the
pullback bundles $f^*(\pi)$ and $g^*(\pi)$ over $B$ are isomorphic.
The same issue arises in the smooth setting.
We use results on the homological algebra of diffeological vector spaces~\cite{Wu}
to give examples of this phenomenon that are unique to the smooth setting.
In these new examples, the group is a diffeological vector space.
We also show that a topological example~\cite{M1,M2} adapts to diffeological
spaces.
In all of these cases, the approach is to give a non-trivial principal
bundle $\pi$ over a smoothly contractible space $X$.
It then follows that the identity map $X \to X$ and a constant map $X \to X$
are smoothly homotopic but that the pullbacks of $\pi$ along these maps are not isomorphic.
It also follows that there is no classifying space for such principal bundles.

Because of this, we focus on $D$-numerable bundles, the smooth
analogs of numerable bundles.  These are bundles for which one can choose
a smooth partition of unity on the base space subordinate to a trivializing open cover.
Our first substantial result is the following:

\theoremstyle{plain}
\newtheorem*{cor:homotopy-pullback}{Corollary \ref{cor:homotopy-pullback}}
\begin{cor:homotopy-pullback}
If $\pi : E' \to B'$ is a $D$-numerable principal $G$-bundle,
and $f$ and $g$ are smoothly homotopic maps $B \to B'$,
then the pullbacks $f^*(\pi)$ and $g^*(\pi)$ are isomorphic as principal $G$-bundles over $B$.
\end{cor:homotopy-pullback}

Here $G$ is a diffeological group, which is a generalization of a Lie group.
Our method of proof follows~\cite{Hu} in outline, but requires many
changes in the details due to the smoothness requirement.
The main technical difficulty is surmounted using the following
result,\footnote{We thank Chengjie Yu for a sketch of the proof of
Proposition~\ref{prop:functional-testing-zero}, and 
Gord Sinnamon and Willie Wong for ideas that led towards this result.}
which may be of independent interest:

\newtheorem*{prop:functional-testing-zero}{Proposition \ref{prop:functional-testing-zero}}
\begin{prop:functional-testing-zero}
 There exists a smooth map $F:C^\infty(\R,\R^{\geq 0}) \to \R^{\geq 0}$ such that 
 $F(f)=0$ if and only if $f(x)=0$ for some $x \in [0,1]$.
\end{prop:functional-testing-zero} 

This function is used in place of the function sending $f$ to $\min_{x \in [0,1]} f(x)$,
which is not smooth.  

Next, given a diffeological group $G$,
we define a principal $G$-bundle $EG \to BG$, show that it is $D$-numerable,
and prove our main result:

\newtheorem*{thm:classify-principal}{Theorem \ref{thm:classify-principal}}
\begin{thm:classify-principal}
For any diffeological space $B$ and any diffeological group $G$, the pullback operation
gives a bijection $[B,BG] \to \Prin_G^D(B)$ which is natural in $B$.
\end{thm:classify-principal}

Here $[B, BG]$ denotes the set of smooth homotopy classes of maps,
and $\Prin_G^D(B)$ denotes the set of isomorphism classes of $D$-numerable
principal $G$-bundles over $B$.
Our $EG$ and $BG$ are set-theoretically the same as those of~\cite{Mi}
and~\cite{Hu}, but are endowed with diffeologies.
When the underlying topological group $D(G)$ is locally compact Hausdorff,
the space $D(BG)$ is homeomorphic to the usual classifying
space of $D(G)$.
Note that because $\pi : EG \to BG$ is itself $D$-numerable and the base $B$ is
not constrained, $\pi$ is truly universal and therefore $BG$ is the unique
diffeological space up to smooth homotopy equivalence that classifies 
$D$-numerable principal $G$-bundles.

We go on to develop the theory of associated bundles, showing
in Theorem~\ref{thm:classify-diff-bundles} that for any diffeological
space $F$, $B\Diff(F)$ classifies $D$-numerable diffeological bundles
with fiber $F$.
Here $\Diff(F)$ is the diffeological group of diffeomorphisms from $F$
to itself, and $D$-numerable diffeological bundles are the smooth
analog of numerable fiber bundles.
We show in Theorem~\ref{thm:naturality-G} that the bijection in
Theorem~\ref{thm:classify-principal} is natural in $G$.

Finally, we define diffeological vector bundles and show
in Theorem~\ref{thm:classify-vb} that for any diffeological
vector space $V$, $B \GL(V)$ classifies the $D$-numerable vector bundles
with fiber $V$.
Here $\GL(V)$ is the diffeological group of smooth linear isomorphisms
from $V$ to $V$.

\smallskip

As mentioned above, many of our arguments follow the topological arguments
in their overall strategy, but differ in the details.
We also adapt a topological result from~\cite{B} in order to correct some
minor gaps in the arguments of~\cite{Hu}.

\smallskip

\textbf{Future work}: 
This paper is intended to provide a foundation for future work.
For example, the theory of characteristic classes for bundles over
a smooth manifold $M$ has two incarnations.  One can use Chern-Weil theory
to construct explicit de Rham forms on $M$ using invariant polynomials
and a connection on the bundle.
Alternatively, one can study the singular cohomology of the classifying
space $BG$, and pull back singular cohomology classes along the classifying map
$M \to BG$.
By the results of the present paper, $BG$ is a diffeological space, and
so one can work directly with de Rham forms on $BG$.
We intend to explore the theory of connections on diffeological bundles,
and use this to apply Chern-Weil theory to the universal case, thereby
bringing the geometric and topological approaches to characteristic
classes closer together.
(See also the remarks below about~\cite{Mo}.)

We also expect the results of this paper to be useful in the study
of smooth tangent bundles and smooth $K$-theory.

\smallskip

\textbf{Relationship to other work}: 
In~\cite{I1}, Iglesias-Zemmour introduced diffeological bundles, 
an elegant generalization of smooth fiber bundles for manifolds, 
well-suited to the category of diffeological spaces. 
Our current work is based on the general theory of diffeological bundles 
established in that thesis.

In~\cite{Mo}, Mostow defined smooth versions of classifying spaces of
Lie groups, using a framework called differentiable spaces.
His focus was on studying the cohomology of such classifying spaces,
and so he did not prove analogs of our results showing that these
spaces do indeed classify certain bundles.
His results are related to the ideas described under the \emph{Future Work}
heading above.
We expect to get cleaner and more general results by working with diffeological
spaces, since the theory of diffeological spaces is better developed.
Moreover, the results of the present paper, which show that the universal
bundle is truly universal, would then complete the circle, giving a
close relationship between the Chern-Weil approach to classifying spaces
and the topological approach.

In~\cite{MW}, Magnot and Watts have independently worked on smooth
classifying spaces using diffeological spaces,
and some comments comparing the approaches are in order.
As sets, our $EG$ and $BG$ are the same as the sets defined by Magnot and Watts,
which we'll denote $EG_{MW}$ and $BG_{MW}$.
However, the diffeologies we use have fewer plots, which leads to better properties.
First, our universal bundle is $D$-numerable, while the MW universal
bundle is only weakly $D$-numerable (see~\cite[Definition 2.17]{MW}).
Moreover, our universal bundle classifies $D$-numerable principal bundles over all diffeological spaces, 
while $EG_{MW} \to BG_{MW}$ only classifies $D$-numerable principal bundles over 
diffeological spaces that are Hausdorff, second-countable and smoothly paracompact.
This greater generality is useful in practice, as one of the aims of diffeological
spaces is to encompass mapping spaces and quotients,
and also means that our classifying
space is uniquely determined up to smooth homotopy equivalence, while $BG_{MW}$ is not.
We also obtain stronger results about the classification of fiber bundles.
Magnot and Watts discuss many topics we do not, such as connections and various applications.

\smallskip

\textbf{Organization}: 
In Section~\ref{se:background}, we review diffeological spaces,
diffeological groups, diffeological bundles and principal bundles.
In Section~\ref{se:no-classifying}, we give examples of non-trivial
principal bundles over smoothly contractible base spaces,
motivating our focus on $D$-numerable bundles.
In Section~\ref{se:D-numerable}, we develop the theory of smooth partitions of unity,
$D$-numerable diffeological bundles, and $D$-numerable principal bundles,
and prove that, for $D$-numerable bundles, homotopic maps give isomorphic pullbacks.
In Section~\ref{se:classify-Dpb}, we define $EG \to BG$ and prove that it is
a universal $D$-numerable bundle, our main result,
using many of the tools from Section~\ref{se:D-numerable}.
In Sections~\ref{se:classify-bundle} and~\ref{se:classify-vb},
we develop the theories of associated bundles and diffeological vector bundles, respectively.
In Appendix~\ref{se:zeros}, we prove needed results in analysis, including
Proposition~\ref{prop:functional-testing-zero}.

\textbf{Conventions}:
Every manifold is assumed to be finite-dimensional, smooth, second-countable, Hausdorff and without 
boundary. Every manifold is equipped with the standard diffeology when viewed as a diffeological space.
Every product of diffeological spaces is equipped with the product diffeology.

\section{Background on diffeological spaces and bundles}
\label{se:background}

\subsection{Diffeological spaces}

\begin{de}[\cite{So2}]\label{de:diffeological-space}
A \dfn{diffeological space} is a set $X$
together with a specified set of functions $U \to X$ (called \dfn{plots})
for each open set $U$ in $\R^n$ and each $n \in \N$,
such that for all open subsets $U \subseteq \R^n$ and $V \subseteq \R^m$:
\begin{enumerate}
\item (Covering) Every constant function $U \to X$ is a plot.
\item (Smooth Compatibility) If $U \to X$ is a plot and $V \to U$ is smooth,
then the composite $V \to U \to X$ is also a plot.
\item (Sheaf Condition) If $U=\cup_i U_i$ is an open cover
and $U \to X$ is a function such that each restriction $U_i \to X$ is a plot,
then $U \to X$ is a plot.
\end{enumerate}

A function $f:X \rightarrow Y$ between diffeological spaces is
\dfn{smooth} if for every plot $p:U \to X$ of $X$,
the composite $f \circ p$ is a plot of $Y$.
\end{de}

An isomorphism in the category $\Diff$ of diffeological spaces and smooth maps will be called a \dfn{diffeomorphism}.

Every manifold $M$ is canonically a diffeological space with the
plots taken to be all smooth maps $U \to M$ in the usual sense.
We call this the \dfn{standard diffeology} on $M$.
It is easy to see that smooth maps in the usual sense between
manifolds coincide with smooth maps between them with the standard diffeology. 

For a diffeological space $X$ with an equivalence relation~$\sim$,
the \dfn{quotient diffeology} on $X/{\sim}$ consists of all functions
$U \to X/{\sim}$ that locally factor through the quotient map $X \to X/{\sim}$ via plots of $X$.
A \dfn{subduction} is a map diffeomorphic to a quotient map.
That is, it is a map $X \to Y$ such that the plots in $Y$
are the functions that locally lift to $X$ as plots in $X$.

For a diffeological space $Y$ and a subset $A$ of $Y$,
the \dfn{sub-diffeology} consists of all functions $U \to A$ such that 
$U \to A \hookrightarrow Y$ is a plot of $Y$.
An \dfn{induction} is an injective smooth map $A \to Y$ such that a
function $U \to A$ is a plot of $A$ if and only if $U \to A \to Y$ is a plot of $Y$.

More generally, we have the following convenient properties of the category of diffeological spaces:

\begin{thm}
The category $\Diff$ is complete, cocomplete and cartesian closed.
\end{thm}

For more details, see~\cite[Section~2]{CSW}. The (co)limit of a diagram of diffeological 
spaces has as its underlying set the (co)limit of the underlying sets 
of the diffeological spaces in the diagram.
Given diffeological spaces $X$ and $Y$, 
the set $C^\infty(X,Y)$ of all smooth maps $X \to Y$ has a canonical diffeology so that the exponential 
law holds.

Every diffeological space has a canonical topology:

\begin{de}[\cite{I1}]
Given a diffeological space $X$, a subset $A \subseteq X$ is \dfn{$D$-open} if
$p^{-1}(A)$ is open in $U$ for each plot $p : U \to X$.
The $D$-open sets form a topology on $X$ called the \dfn{$D$-topology},
and we write $D(X)$ for the set $X$ equipped with this topology.
\end{de}

\begin{ex}
The $D$-topology of a manifold with the standard diffeology is the usual topology.
\end{ex}

\begin{rem}\label{rem:disjoint-union}
If $X$ is a disjoint union of $D$-open subsets $U_i$, then $X$ is the coproduct of
the $U_i$ in the category of diffeological spaces.
\end{rem}

\subsection{Diffeological bundles}

\begin{de}\label{de:bundles}
Let $F$ be a diffeological space.
A smooth map $\pi:E \to B$ between two diffeological spaces
is \dfn{trivial of fiber type $F$}
if there exists a diffeomorphism $h$
making the following diagram commute:
\[
\xymatrix@C5pt{E \ar[dr]_\pi \ar[rr]^-h && B \times F\  \ar[dl]^{p_1} \\ & B,}
\]
where $p_1$ is the projection.

The map $\pi$ is \dfn{locally trivial of fiber type $F$}
if there exists a $D$-open cover $\{ B_i \}$ of $B$ such that
$\pi|_{B_i}:\pi^{-1}(B_i) \to B_i$ is trivial of fiber type $F$ for each $i$.

The map $\pi$ is a \dfn{diffeological bundle of fiber type $F$} if the pullback of $\pi$ along any plot
of $B$ is locally trivial of fiber type $F$.

In all of these cases, we call $F$ the \dfn{fiber} of $\pi$, $E$ the \dfn{total space}, and $B$ the \dfn{base space}.

Two diffeological bundles $\pi:E \to B$ and $\pi':E' \to B$ are \dfn{isomorphic} if there exists a diffeomorphism 
$h:E \to E'$ such that $\pi = \pi' \circ h$.
\end{de}

The concept of diffeological bundle was first defined in~\cite{I1}
using groupoids.  Proposition~3.1.2 of~\cite{I1} shows that it is equivalent
to the definition given above.  Moreover, Iglesias-Zemmour has given
another equivalent characterization of diffeological bundles:

\begin{thm}[{\cite[Corollary~3.8.3]{I1} or \cite[8.19]{I2}}]
A smooth map $\pi:E \to B$ between two diffeological spaces
is a diffeological bundle of fiber type $F$ if and only if
the pullback of $\pi$ along any global plot of $B$
(that is, a plot of the form $\R^n \to B$) is trivial of fiber type $F$.
\end{thm}

Note that every locally trivial bundle is a diffeological bundle, but that
the converse fails in general.

\begin{ex}
If $B$ is a manifold and $\pi : E \to B$ is smooth, then $\pi$ is locally trivial
of fiber type $F$ if and only if it is a diffeological bundle of fiber type $F$.
Moreover, if the fiber $F$ is a manifold, then it is also equivalent for $\pi$
to be a smooth fiber bundle in the usual sense.
\end{ex}

\begin{lem}\label{lem:disjoint-union}
If $B$ is a disjoint union of $D$-open sets $B_i$ and $\pi: E \to B$ is a diffeological
bundle which is trivial over each $B_i$, then $\pi$ is trivial.
\end{lem}

\begin{proof}
$E$ is the disjoint union of the open sets $E_i := \pi^{-1}(B_i)$, so
by Remark~\ref{rem:disjoint-union}, $E \to B$ is the coproduct of the
trivial bundles $E_i \to B_i$ and hence is trivial.
\end{proof}

\subsection{Principal bundles}

\begin{de}[\cite{So1}]
A \dfn{diffeological group} is a group object in $\Diff$.
That is, a diffeological group is both a diffeological space and a group
such that the group operations are smooth maps.
\end{de}

\begin{ex}
Every subgroup of a diffeological group equipped with the sub-diffeology is a diffeological group.
\end{ex}

\begin{ex}
Given a diffeological space $X$, write $\Diff(X)$ for the set of all diffeomorphisms 
$X \to X$. Define $p:U \to \Diff(X)$ to be a plot if the maps $U \times X \to X$ given by 
$(u,x) \mapsto p(u)(x)$ and $(u,x) \mapsto (p(u))^{-1}(x)$ are both smooth. These plots
form a diffeology on $\Diff(X)$ making it a diffeological group.
We always equip $\Diff(X)$ with this diffeology.
\end{ex}

Here is a $G$-equivariant version of Definition~\ref{de:bundles}:

\begin{de}
Let $G$ be a diffeological group and let $\pi : E \to B$ be a smooth map
between diffeological spaces.
Assume that $G$ has a \dfn{smooth right action} on $E$, i.e., $E$ has a
right $G$-action and the action map
$E \times G \to E$ is smooth.  Also assume that $\pi(x \cdot g) = \pi(x)$
for all $x \in E$ and $g \in G$.

We say that $\pi$ is a \dfn{trivial principal $G$-bundle} if there is a $G$-equivariant
diffeomorphism $h$ making the following diagram commute:
\[
\xymatrix@C5pt{E \ar[dr]_\pi \ar[rr]^-h && B \times G\  \ar[dl]^{p_1} \\ & B.}
\]
Here the action of $G$ on $B \times G$ is defined by $(b, g) \cdot g' = (b, g g')$.

We say that $\pi$ is a \dfn{locally trivial principal $G$-bundle}
if there exists a $D$-open cover $\{ B_i \}$ of $B$ such that
$\pi|_{B_i}:\pi^{-1}(B_i) \to B_i$ is a trivial principal $G$-bundle for each $i$.

The map $\pi$ is a \dfn{(diffeological) principal $G$-bundle} if the pullback of $\pi$ along any plot
of $B$ is a locally trivial $G$-bundle.

Two principal $G$-bundles $\pi:E \to B$ and $\pi':E' \to B$ are \dfn{isomorphic} if there exists a $G$-equivariant 
diffeomorphism $h:E \to E'$ such that $\pi = \pi' \circ h$. 
\end{de}

Here is an equivalent characterization of principal bundles which will be used frequently later:

\begin{thm}[{\cite[8.11, 8.13]{I2}}]\label{thm:principal}
 If $E \to B$ is a principal $G$-bundle, then the smooth map
 $a : E \times G \to E \times E$ given by $(x,g) \mapsto (x,x \cdot g)$ is an induction
 and there is a diffeomorphism $B \cong E/G$ commuting with the maps from $E$.
 Conversely, if $E \times G \to E$ is a smooth action of a diffeological group $G$ on $E$
 and the map $a$ is an induction, then the quotient map $E \to E/G$ is a principal $G$-bundle.
\end{thm}

\begin{rem}\label{rem:trivial-pb}
It follows from the above theorem that a principal bundle is trivial if and only if it has a smooth global section~\cite[8.12]{I2}.
Therefore, a principal bundle is locally trivial as a principal bundle if and only if
it is locally trivial as a diffeological bundle.
\end{rem}

As another application of the above theorem, we have:

\begin{prop}[{\cite[8.15]{I2}}]\label{prop:G->G/H-principal-bundle}
Let $G$ be a diffeological group, and let $H$ be a subgroup of $G$ with the sub-diffeology.
Then $G \to G/H$ is a principal $H$-bundle,
where $G/H$ is the set of left cosets of $H$ in $G$, with the quotient diffeology.
\end{prop}

Note that we are \emph{not} requiring the subgroup $H$ to be closed.
In particular, we have the following interesting example:

\begin{ex}[{\cite[8.38]{I2}}]\label{ex:irrational-torus}
Let $T^2=\R^2/\Z^2$ be the usual $2$-torus,
and let $\R_\theta$ be the image of the line $\{ y = \theta x \}$
under the quotient map $\R^2 \to T^2$, with $\theta$ a fixed irrational number.
Note that $T^2$ is an abelian Lie group,
and $\R_\theta$ is a dense subgroup which is diffeomorphic to $\R$.
The quotient group $T^2_\theta := T^2/\R_\theta$ with the quotient diffeology
is called the \dfn{irrational torus of slope $\theta$}, and
by the above proposition,
the quotient map $T^2 \to T^2_\theta$ is a principal $\R$-bundle.
\end{ex}

\begin{prop}[{\cite[8.10, 8.12]{I2}}]\label{prop:pullback}
If $f:B' \to B$ is a smooth map and $E \to B$ is a diffeological
bundle of fiber type $F$ (resp.\ a principal $G$-bundle),
then so is the pullback $f^*(E) \to B'$.
Moreover, pullback preserves triviality and local triviality.
\end{prop}

The following result follows immediately from~\cite[8.13 Note 2]{I2}
and will be useful later:

\begin{prop}\label{prop:commsq=>pullbackdiff}
 Let 
 \[
  \xymatrix{E' \ar[r]^f \ar[d]_{\pi'} & E \ar[d]^{\pi} \\ B' \ar[r]_g & B}
 \]
 be a commutative square in $\Diff$, where $\pi'$ and $\pi$ are principal $G$-bundles
 and $f$ is $G$-equivariant. Then $\pi'$ is isomorphic to $g^*(\pi)$ as principal $G$-bundles over $B'$.
\end{prop}

\section{There is no classifying space for all diffeological principal bundles}
\label{se:no-classifying}

This section motivates our focus on $D$-numerable bundles in later sections.

We first recall the situation in topology.
Let $G$ be a topological group.
We would like to have a space $BG$ such that for any topological space $X$,
the set of isomorphism classes of principal $G$-bundles over $X$ naturally 
bijects with the set of homotopy classes of maps from $X$ to $BG$.
This is possible when the space $X$ is restricted to being paracompact,
or, more generally, if one considers only numerable bundles, but is
not possible in general.
One way to show that it is not possible is as follows.
First observe that $BG$ must be path-connected, by taking the case where $X$ is a point.
Next, one shows that there is a non-trivial
principal $G$-bundle $\pi$ over a contractible space $X$.  Then there are at
least two non-isomorphic principal $G$-bundles over $X$, but only one homotopy
class of maps $X \to BG$ for any path-connected space $BG$.
In addition, such an example shows that, in general, homotopic maps do not
have isomorphic pullbacks:  the pullback of $\pi$ along the identity map
$X \to X$ is $\pi$, while the pullback of $\pi$ along a constant map is trivial.

Analogous results hold in the diffeological context, and the same technique is used.
In the first part of this section, we give a family of examples of non-trivial
diffeological principal bundles over smoothly contractible base spaces, using the
theory of diffeological vector spaces from~\cite{Wu}.
These examples are new, and we are not aware of similar topological examples.

Then, in Example~\ref{ex:Goodwillie}, we give another example of a non-trivial
principal bundle over a smoothly contractible base space.
This example is even locally trivial, and so shows that restricting to this
subclass of diffeological principal bundles does not solve the problem.
This example is a straightforward adaptation of an example from topology~\cite{M1,M2}.

We begin by recalling the concept of smooth homotopy~\cite[Section 3.1]{CW1}:

\begin{de}
Given diffeological spaces $X$ and $Y$, two smooth maps $f,g:X \to Y$ are called \dfn{smoothly homotopic} 
if there exists a smooth map $F: X \times \R \to Y$ such that $F(x, 0)=f(x)$ and $F(x, 1)=g(x)$ for each $x$ in $X$.
A diffeological space $X$ is \dfn{smoothly contractible} if the identity map
is smoothly homotopic to a constant map.
\end{de}

Given diffeological spaces $X$ and $Y$, the relation of smooth homotopy on $C^\infty(X,Y)$ is an equivalence relation, 
and we denote the quotient set by $[X,Y]$.

\begin{de}
A \dfn{diffeological vector space} $V$ is both a diffeological space and an $\R$-vector space such that addition $V \times V \to V$ 
and scalar multiplication $\R \times V \to V$ are both smooth.
\end{de}

Observe that every diffeological vector space $V$ is smoothly contractible
via the smooth homotopy sending $(v, t)$ to $t v$.
A \dfn{short exact sequence} in the category $\DVect$ of diffeological vector spaces and smooth linear maps
is a diagram
\begin{equation}\label{ses}
\xymatrix{0 \ar[r] & V_1 \ar[r]^i & V_2 \ar[r]^j & V_3 \ar[r] & 0}
\end{equation}
which is a short exact sequence of vector spaces such that $i$ is a linear induction
and $j$ is a linear subduction.  For any such short exact sequence, we have a commutative triangle
\[
\xymatrix@C5pt{& V_2 \ar[dl]_j \ar[dr]^\pi \\ V_3 && V_2/V_1, \ar[ll]}
\]
where the horizontal map is an isomorphism of diffeological vector spaces. 
Hence, by Proposition~\ref{prop:G->G/H-principal-bundle}, $j$ is a diffeological principal $V_1$-bundle. 
This bundle is trivial if and only if the short exact sequence~\eqref{ses} splits smoothly (see~\cite[Theorem~3.16]{Wu}).
In particular, it follows that if~\eqref{ses} does not split smoothly,
then there is no classifying space for principal $V_1$-bundles.

\begin{ex}\label{ex:Borel}
Let $j : C^\infty(\R, \R) \to \prod_{\omega} \R$ be defined by $j(f)_n := f^{(n)}(0)$,
and let $K$ be the kernel.  It is shown in~\cite[Example~4.3]{Wu} that this is
a short exact sequence of diffeological vector spaces that does not split smoothly.
Therefore, there is no classifying space for principal $K$-bundles.
\end{ex}

\medskip

We now give additional examples of this flavour, using some results
from~\cite{Wu}.

\begin{de}
A diffeological vector space $P$ is called \dfn{projective} if for any linear subduction $\pi:W_1 \to W_2$ and any 
smooth linear map $f:P \to W_2$, there exists a smooth linear map $g:P \to W_1$ such that $f = \pi \circ g$.
\end{de}

\begin{prop}[{\cite[Theorem~6.13]{Wu}}]
For every diffeological vector space $V$, there exists a projective
diffeological vector space $P$ with a linear subduction $P \to V$.
\end{prop}

\begin{thm}\label{thm:non-projective}
Let $V$ be a non-projective diffeological vector space.
Then there exists a diffeological vector space $W$ and 
a non-trivial diffeological principal $W$-bundle over $V$.
This implies that there is no classifying space for principal $W$-bundles.
\end{thm}

\begin{proof}
Let $W$ be the kernel of a linear subduction $P \to V$ with $P$ projective.
Since projectives are closed under summands (\cite[Proposition~6.11(3)]{Wu}),
the sequence does not split smoothly.
Thus there is a non-trivial principal $W$-bundle over $V$.
\end{proof}

\begin{ex}
We saw in Example~\ref{ex:Borel} that $\prod_{\omega} \R$ is not projective.
It is shown in~\cite[Example~6.9]{Wu} that $\R$ with the indiscrete\footnote{An
indiscrete diffeological space has all possible functions as plots, and hence has the indiscrete $D$-topology.}
diffeology is also not projective.
\end{ex}

To obtain further examples, including examples where the base space is
not a diffeological vector space, we make use of the following construction.

\begin{prop}[{\cite[Proposition~3.5]{Wu}}]
For every diffeological space $X$, there is a diffeological vector space $F(X)$ together with a smooth map 
$i:X \to F(X)$ such that the following universal property holds: for any diffeological vector space $V$ and any smooth 
map $f:X \to V$, there exists a unique smooth linear map $g:F(X) \to V$ satisfying $f = g \circ i$.
\end{prop}

We call $F(X)$ the \dfn{free diffeological vector space generated by $X$}.

\begin{ex}
Not every free diffeological vector space is projective.
For example, it is shown in~\cite[Example~6.9]{Wu} that $F(T^2_\theta)$
is not projective, where $T^2_\theta$ is the irrational torus from
Example~\ref{ex:irrational-torus}.
\end{ex}

Some necessary conditions for a free diffeological vector space to be projective have been found 
in~\cite[Corollary~3.17]{CW2}.

\begin{cor}
Let $X$ be a diffeological space such that $F(X)$ is not projective. Then there exists a non-trivial diffeological 
principal $W$-bundle over $X$, where $W$ is a diffeological vector space.
\end{cor}

\begin{proof}
By Theorem~\ref{thm:non-projective}, there is a non-trivial diffeological principal $W$-bundle
$\pi:P \to F(X)$ with $W$ a diffeological vector space.
Consider its pullback $p : E \to X$ along $i: X \to F(X)$.
We claim that $p$ is not trivial.
Suppose it is.
Then $p$ has a smooth section, which implies 
that there exists a smooth map $f:X \to P$ such that $i = \pi \circ f$.
The universal property of $F(X)$ then implies that $\pi$ has a smooth section over $F(X)$.
By Remark~\ref{rem:trivial-pb}, we deduce that $\pi$ is trivial, a contradiction.
\end{proof}

\begin{ex}
If $X$ is an indiscrete diffeological space with more than one point,
then $X$ is smoothly contractible and $F(X)$ is not projective (\cite[Corollary~3.17]{CW2}).
So there exists a non-trivial diffeological principal bundle $\pi : E \to X$.
\end{ex}

The above examples are diffeological principal bundles which may not be locally trivial,
and thus have no direct analog in topology.
We now show that even locally trivial diffeological principal bundles do not have
a classifying space.

\begin{ex}\label{ex:Goodwillie}
The following is adapted from a topological example~\cite{M1,M2}.
Consider the diffeological space $B := (\R \times \{0,1\})/{\sim}$, where $(x,0) \sim (x,1)$ if 
$x \in \R^{>0}$. Write $r:\R \times \{0,1\} \to B$ for the quotient map and let $U_i := r(\R \times \{i\})$ for $i=0,1$. 
Then each $U_i$ is $D$-open in $B$ and canonically diffeomorphic to $\R$, 
and $U_0 \cap U_1$ is canonically diffeomorphic to $\R^{>0}$. Define $E$ to be the pushout of 
\[
\xymatrix{U_1 \times \R^{>0} & (U_0 \cap U_1) \times \R^{>0}\ \ar[l] \ar@{^{(}->}[r] & U_0 \times \R^{>0},}
\]
where the first map is given by $(r(x,i),g) \mapsto (r(x,1),xg)$.
The projections $U_i \times \R^{>0} \to U_i \hookrightarrow B$ induce a smooth map $p:E \to B$
which is a locally trivial principal $\R^{>0}$-bundle. 
Here we are regarding $\R^{>0}$ as a diffeological group under multiplication.
Consider the smooth map ${(\R \times \{0,1\}) \times \R} \to \R \times \{0,1\}$ defined by 
$(x,i,t) \mapsto (\rho(t)+(1-\rho(t))x,\, i)$ for $i=0,1$, where $\rho:\R \to \R$ is a smooth function 
with $\rho(0)=0$, $\rho(1)=1$ and $\im(\rho)=[0,1]$.
This induces a smooth homotopy $B \times \R \to B$ between the identity map and a constant map,
which shows that $B$ is smoothly contractible. 
Now if $p:E \to B$ were trivial, we would have an $\R^{>0}$-equivariant trivialization
$E \cong B \times \R^{>0}$ over $B$.  Restricting to each $U_i \times \R^{>0}$ would give maps
$U_i \times \R^{>0} \to B \times \R^{>0}$ sending $(r(x,i),g)$ to $(r(x,i),\alpha_i(x) g)$
for some smooth functions $\alpha_i : \R \to \R^{>0}$.
Since these restrictions must agree on $(U_0 \cap U_1) \times \R^{>0}$,
we must have that $\alpha_0(x) = \alpha_1(x) x$ for $x > 0$.
But then the identity map $\R^{>0} \to \R^{>0}$ would have a smooth extension $\R \to \R^{>0}$
sending $x$ to $\alpha_0(x)/\alpha_1(x)$,
which is impossible by continuity.
\end{ex}

\section{Partitions of unity and $D$-numerable bundles}
\label{se:D-numerable}

In the previous section, we saw that in general there is no classifying space for all principal $G$-bundles.
Because of this, we restrict our attention to a special class of principal bundles called $D$-numerable principal bundles. 
In the next section, we will show that there is a classifying space for $D$-numerable principal bundles over arbitrary diffeological spaces.

\subsection{Partitions of unity}

We first recall the concept of smooth partition of unity in the framework of diffeology:

\begin{de}
 Let $X$ be a diffeological space.

 A collection $\{ U_i \}_{i \in I}$ of subsets of $X$ is \dfn{locally finite} if every point in $X$
 has a $D$-open neighbourhood that intersects $U_i$ for only finitely many $i$.
 Note that $\{ U_i \}$ is locally finite if and only if
 the collection $\{ \bar{U_i} \}$ of $D$-closures is locally finite.

 A family of smooth functions $\{\mu_i:X \to \R\}_{i \in I}$ is a 
 \dfn{smooth partition of unity} if it satisfies the following conditions:
 \begin{enumerate}
  \item $0 \leq \mu_i(x)$ for each $i \in I$ and $x \in X$;
  
  \item $\{\supp(\mu_i)\}_{i \in I}$ is locally finite;
   
  \item the sum $\sum_{i \in I} \mu_i(x)$, which makes sense because of (2), is equal to $1$ for all $x \in X$.
 \end{enumerate}
 Here $\supp(\mu_i)$ is the closure of $\mu_i^{-1}((0, \infty))$ in the $D$-topology,
 and (2) is equivalent to requiring that $\{ \mu_i^{-1}((0, \infty)) \}$ is locally finite.

 If $\mathfrak{X}=\{X_i\}_{i \in I}$ is a collection of subsets of $X$ indexed by the
 same indexing set, we say that our partition of unity is \dfn{subordinate to $\mathfrak{X}$} if
 $\supp(\mu_i) \subseteq X_i$ for each $i \in I$.
\end{de}

If instead of (3), we have that $\sum_{i \in I} \mu_i(x)$ is nonzero for each $x \in X$,
or equivalently that the sets $\mu_i^{-1}((0, \infty))$ form a cover of $X$, then one
can scale the functions to obtain a smooth partition of unity.

The following is a smooth version of a result that can be found in~\cite[Section 4]{B}.
It tells us how to adjust a partition of unity to reduce the supports,
allowing us to fill some minor gaps in the arguments of~\cite{Hu}.

\begin{lem}\label{lem:Bourbaki}
 Let $X$ be a diffeological space.
 If $\{\rho_i:X \to \R\}_{i \in I}$ is a smooth partition of unity,
 then there is a smooth partition of unity $\{ \mu_{i}:X \to \R \}_{i \in I}$ subordinate to $\{\rho_i^{-1}((0,\infty))\}_{i \in I}$.
\end{lem}

\begin{proof}
Define $\sigma : X \to \R$ by $\sigma(x) = \sum_i \rho_i(x)^2$.
Note that $\sigma$ is smooth, nowhere zero and
\[
  \sigma(x) \leq (\sup_i \rho_i(x)) \sum_i \rho_i(x) = \sup_i \rho_i(x)
\]
for each $x$.
Let $\phi$ be a smooth function such that $\phi(t) = 0$ for $t \leq 0$ and $\phi(t) > 0$ for $t > 0$,
and define a smooth function $\mu_i : X \to \R$ by $\mu_i(x) = \phi(\rho_i(x) - \sigma(x)/2)$ for each $i$. 

We will show that $\supp(\mu_i) \subseteq \rho_i^{-1}((0,\infty))$, which then implies that $\{\supp(\mu_i)\}_{i \in I}$ is 
locally finite.
Suppose that $\rho_i(y) = 0$.
Then there is a $D$-open neighbourhood $V$ of $y$ such that $\rho_i(x) - \sigma(x)/2 < 0$
for $x$ in $V$.  That is, $\mu_i(x) = 0$ for each $x$ in $V$.
Therefore, $y \not\in \supp(\mu_i)$, as required.

Since $\{ \supp(\mu_i) \}$ is locally finite, $\sum_i \mu_i(x)$ is well-defined.
Note that for each $x$ there is a $j$ such that 
$\rho_j(x) = \sup_i \rho_i(x) \geq \sigma(x) > \sigma(x)/2$.
For this $j$, $\mu_j(x) \neq 0$, and so $\sum_i \mu_i(x)$ is nowhere zero. 
Therefore the functions
$\mu_i$ can be scaled to form a smooth partition of unity subordinate to $\{\rho_i^{-1}((0,\infty))\}_{i \in I}$.
\end{proof}

Our next lemma shows that one can replace any partition of unity with a 
related countable one.

\begin{lem}\label{lem:countable}
Let $B$ be a diffeological space and
let $\{\rho_i:B \to \R\}_{i \in I}$ be a smooth partition of unity.
Then there exists a countable smooth partition of unity $\{\tau_n:B \to \R\}_{n \in \N}$
such that each $\tau_n^{-1}((0,\infty))$ is a disjoint union of $D$-open sets each of which
is contained in $\rho_i^{-1}((0,\infty))$ for some $i \in I$.
\end{lem}

\begin{proof}
Fix a smooth function 
$\phi:\R \to \R$ with $\phi(t)=0$ if $t \leq 0$ and $\phi(t) > 0$ if $t > 0$.
For any non-empty finite subset $J$ of the indexing set $I$, define $\sigma_J:B \to \R$ by 
$\sigma_J(b) = \prod_{j \in J} \phi(\rho_j(b) - \sum_{k \in I \setminus J} \rho_k(b))$. 
By local finiteness of $\{\supp(\rho_i)\}_{i \in I}$, $\sigma_J$ is well-defined 
and smooth. Write $B_J := \sigma_J^{-1}((0,\infty))$.
Since each $b \in B$ is in $B_J$, where $J = \{ j \in I \mid \rho_j(b) \neq 0 \}$,
we have that $\cup_J \, B_J = B$.  Moreover, each $B_J \subseteq \rho_j^{-1}((0,\infty))$ for 
any $j \in J$.

Write $|J|$ for the cardinality of the set $J$. Then for any $J \neq J'$ with $|J| = |J'|$, 
we have $B_J \cap B_{J'} = \emptyset$. Otherwise, let $b \in B_J \cap B_{J'}$,
and choose $j \in J \setminus J'$ and $j' \in J' \setminus J$.
Since $b \in B_J$, we have that $\rho_j(b) - \sum_{k \in I \setminus J} \rho_k(b) > 0$, 
which implies that $\rho_j(b) > \rho_{j'}(b)$.
But $b \in B_{J'}$ implies that $\rho_j(b) < \rho_{j'}(b)$, a contradiction.

For $n \in \N^{>0}$, define $\tau_n:B \to \R$ by $\tau_n(b) = \sum_{J \subseteq I, |J|=n} \sigma_J(b)$. 
Then $B_n := \tau_n^{-1}((0,\infty)) = \cup_{J \subseteq I, |J|=n} \, B_J$ is a disjoint union
of sets $B_J$ each of which is contained in some $\rho_j^{-1}((0,\infty))$.
(Also define $\tau_0$ to be the zero function, with $B_0 = \emptyset$.)
By local finiteness of $\{\rho_i^{-1}(0,\infty)\}_{i \in I}$, one sees that $\{B_n\}_{n \in \N}$ is locally finite,
and therefore that $\{ \supp(\tau_n) \}_{n \in \N}$ is locally finite.
The result then follows by normalizing the $\tau_n$'s.
\end{proof}

\begin{samepage}
\subsection{$D$-numerable diffeological bundles}

\begin{de}
 Let $F$ be a diffeological space. A smooth map $\pi:E \to B$ is called a \dfn{$D$-numerable diffeological bundle of fiber type $F$} if there 
 exists a smooth partition of unity $\{\mu_i:B \to \R\}_{i \in I}$ subordinate to a $D$-open cover $\{B_i\}_{i \in I}$ of $B$ such that each 
 $\pi|_{B_i}$ is trivial of fiber type $F$.
\end{de}
\end{samepage}

Clearly, 
\[
\text{trivial} \implies \text{$D$-numerable} \implies \text{locally trivial} \implies \text{diffeological bundle}.
\]
By Lemmas~\ref{lem:Bourbaki} and~\ref{lem:countable}, our definition of $D$-numerable agrees with that of~\cite{MW}.

\begin{ex}
If $B$ is a manifold, then the following concepts (over $B$) coincide:
\begin{enumerate}
\item $D$-numerable diffeological bundle;
\item locally trivial bundle;
\item diffeological bundle.
\end{enumerate}
\end{ex}

\begin{ex}
If a diffeological space $B$ has indiscrete $D$-topology, then the only $D$-numerable diffeological bundle over $B$ is the trivial 
bundle. In particular, the only $D$-numerable diffeological bundle over an irrational torus or an indiscrete diffeological space is trivial.
\end{ex}

\begin{lem}\label{lem:pullback-D-numerable}
The pullback of a $D$-numerable diffeological bundle of fiber type $F$ is again $D$-numerable of fiber type $F$.
\end{lem}

\begin{proof}
 This is straightforward.
\end{proof}

We now show that one can assume that the indexing set is countable.

\begin{prop}\label{prop:countable-diff}
Let $\pi:E \to B$ be a $D$-numerable diffeological bundle.
Then there exists a countable smooth partition of unity $\{\mu_n:B \to \R\}_{n \in \N}$ subordinate to 
a locally finite $D$-open cover $\{B_n\}_{n \in \N}$ of $B$ such that 
$\pi|_{B_n}:\pi^{-1}(B_n) \to B_n$ is trivial for each $n$.
\end{prop}

\begin{proof}
Let $\{\rho_i:B \to \R\}_{i \in I}$ be a smooth partition of unity subordinate to a 
$D$-open cover $\{U_i\}_{i \in I}$ of $B$ such that each 
$\pi|_{U_i}$ is a trivial diffeological bundle.
By Lemma~\ref{lem:countable}, there is a countable smooth partition of unity
$\{ \tau_n : B \to \R \}_{n \in \N}$ such that each $B_n := \tau_n^{-1}((0, \infty))$ is a disjoint
union of $D$-open sets each of which is contained in $\rho_i^{-1}((0, \infty))$ for some $i$.
It follows from Lemma~\ref{lem:disjoint-union} that $\pi|_{B_n} : \pi^{-1}(B_n) \to B_n$ is trivial for each $n$.
By Lemma~\ref{lem:Bourbaki}, we can find another countable smooth partition of unity
$\{ \mu_n : B \to \R \}_{n \in \N}$ subordinate to $\{ B_n \}_{n \in \N}$,
which completes the argument.
\end{proof}

\subsection{$D$-numerable principal bundles}

\begin{de}
 Let $G$ be a diffeological group. A principal $G$-bundle $\pi:E \to B$ is \dfn{$D$-numerable} if there 
 exists a smooth partition of unity $\{\mu_i:B \to \R\}_{i \in I}$ subordinate to a $D$-open cover $\{B_i\}_{i \in I}$ of $B$ such that each 
 $\pi|_{B_i}$ is a trivial principal $G$-bundle.
\end{de}

By Remark~\ref{rem:trivial-pb}, it is equivalent to require that $\pi$ is
$D$-numerable as a diffeological bundle.

Just as for diffeological bundles, we can assume that the indexing set is countable.
This will be used in the proof of Proposition~\ref{prop:theta-surjective}.

\begin{prop}\label{prop:countable}
Let $\pi:E \to B$ be a $D$-numerable principal $G$-bundle. 
Then there exists a countable smooth partition of unity $\{\mu_n:B \to \R\}_{n \in \N}$ subordinate to 
a locally finite $D$-open cover $\{B_n\}_{n \in \N}$ of $B$ such that 
$\pi|_{B_n}:\pi^{-1}(B_n) \to B_n$ is trivial for each $n$.
\end{prop}

\begin{proof}
This follows from Proposition~\ref{prop:countable-diff}.
\end{proof}\pagebreak[2]%

Our next goal is to show that pulling back a $D$-numerable principal bundle along
homotopic maps gives isomorphic bundles.
While the general argument follows existing approaches from topology,
several key steps need novel proofs in order to work in the smooth setting.

\begin{lem}\label{lem:pullback-D-numerable-G}
The pullback of a $D$-numerable principal $G$-bundle is a $D$-numerable principal $G$-bundle.
\end{lem}

\begin{proof}
 This is straightforward.
\end{proof}

\begin{prop}\label{prop:Dpb-on-product}
 For every $D$-numerable principal $G$-bundle $\pi:E \to B \times \R$, there exists a $D$-open cover 
 $\{B_k\}_{k \in K}$ of $B$ together with a smooth partition of unity subordinate to it such that 
 $\pi|_{B_k \times [0,1]}:\pi^{-1}(B_k \times [0,1]) \to B_k \times [0,1]$ is trivial for each $k \in K$.
\end{prop}

This proof is based on the proof of~\cite[Lemma~4.9.5]{Hu} (which is essentially due to~\cite{Do}), with
the function $F$ from Proposition~\ref{prop:functional-testing-zero}
playing the role of the $\min$ function.

\begin{proof}
 Let $\{\rho_i:B \times \R \to \R\}_{i \in I}$ be a smooth partition of unity 
 such that $\pi$ is trivial over each set $\rho_i^{-1}((0, \infty))$.
 By Proposition~\ref{prop:functional-testing-zero}, there exists
 a smooth map $F:C^\infty(\R,\R^{\geq 0}) \to \R^{\geq 0}$ such that 
 $F(f)=0$ if and only if $f(s)=0$ for some $s \in [0,1]$.
 For every $n \in \Z^{>0}$ and $k=(k(1),\ldots,k(n)) \in I^n$, define 
 $\hat{\rho}_k:B \to \R$ by $b \mapsto \prod_{i=1}^n F(\tilde{\rho}_i(b))$, where
 $\tilde{\rho}_i:B \to C^\infty(\R,\R^{\geq 0})$ is defined by 
 $\tilde{\rho}_i(b)(s)=\rho_{k(i)}(b,\frac{2s+i-3/2}{n})$, using cartesian closedness of $\Diff$.
 Write $B_k := \hat{\rho}_k^{-1}((0,\infty))$, which is $D$-open in $B$ since $\hat{\rho}_k$ is smooth.
Then
 $b \in B_k$ if and only if $\{b\} \times [\frac{i-3/2}{n},\frac{i+1/2}{n}] \subseteq \rho_{k(i)}^{-1}((0, \infty))$ 
 for each $i \in \{1,2,\ldots,n\}$,
 which implies that $\pi$ is trivial on each $B_k \times (\frac{i-3/2}{n},\frac{i+1/2}{n})$.
 By~\cite[Lemma~1 in 8.19]{I2}, we see that $\pi|_{B_k \times [0,1]}:\pi^{-1}(B_k \times [0,1]) \to B_k \times [0,1]$ is
 trivial.
 
 Let $K = \cup_n I^n$ and
 write $\mathfrak{B}=\{B_k\}_{k \in K}$. Since $[0,1]$ is compact, it is easy to see that for every 
 $b \in B$, there exists $l \in K$ such that $b \in B_l$, i.e., $\mathfrak{B}$ is a $D$-open cover of $B$. 
 By~\cite[Lemma~4.1]{CSW}, the $D$-topology on $B \times \R$ coincides with the product topology.
 Fix $b \in B$ and $n \in \N$.
 For $i = 1, \ldots, n$, there exist $D$-open sets $U_i \subseteq B$ and $V_i \subseteq \R$
 such that $(b, i/n) \in U_i \times V_i$ and $U_i \times V_i$ intersects only finitely many
 of the sets $\rho_j^{-1}((0,\infty))$ for $j \in I$.
 Let $U := \cap_{i=1}^n \, U_i$, so the same properties hold for each $U \times V_i$.
 For $k \in I^n$, $b \in B_k$ implies that $(b, i/n) \in 
 \{ b \} \times [\frac{i-3/2}{n},\frac{i+1/2}{n}] \subseteq \rho_{k(i)}^{-1}((0, \infty))$ 
 for each $i \in \{1,2,\ldots,n\}$,
 and so there are only finitely many $k \in I^n$ so that $U$ intersects $B_k$.

 We next tweak the functions in order to make their supports locally finite as $n$ varies as well.
 For each $r \in \N^{>1}$, write $\tau_r$ for the sum of all $\hat{\rho}_{k'}$ with $k' \in I^n$ and $n<r$, 
 and write $\tau_0 = \tau_1 = 0$.
 Each $\tau_r : B \to \R$ is smooth, by the previous paragraph.
 Fix a smooth function $\phi:\R \to \R$ with $\phi(t)=0$ for all $t \leq 0$ and 
 $\phi(t)>0$ for all $t>0$. For $k \in I^r$, define $\sigma_k:B \to \R$ by 
 $\sigma_k(b)=\phi(\hat{\rho}_k(b) - r \tau_r(b))$.
 For fixed $b \in B$, we have a $\bar{k} \in I^{\bar{r}}$ with $\bar{r}$ minimal with respect to 
 the property that $\hat{\rho}_{\bar{k}}(b)>0$. From this, one obtains that 
 $\sigma_{\bar{k}}(b)=\phi(\hat{\rho}_{\bar{k}}(b) - \bar{r} \tau_{\bar{r}}(b)) = \phi(\hat{\rho}_{\bar{k}}(b)) > 0$. 
 On the other hand, let $m \in \N$ be such that $m > \bar{r}$ and $\hat{\rho}_{\bar{k}}(b) > 1/m$.
 Since $\hat{\rho}_{\bar{k}}:B \to \R$ is smooth, there exists a $D$-open neighborhood $V$ of $b$ such that 
 for every $x \in V$, $\hat{\rho}_{\bar{k}}(x) > 1/m$. Then for any $l \geq m$, we have $l \tau_l(x) \geq m \tau_m(x) \geq m \hat{\rho}_{\bar{k}}(x) > 1$
 for all $x \in V$,
 i.e., $\sigma_k(x)=0$ for all $k \in I^l$ and $x \in V$. 
 Therefore, $\{\sigma_k^{-1}((0,\infty))\}_{k \in K}$ is locally finite.
 
 Therefore, after scaling, the conditions in Lemma~\ref{lem:Bourbaki} hold for $\{\sigma_k\}_k$, and we get a smooth partition of unity 
 subordinate to $\{\sigma_k^{-1}((0,\infty))\}_k$. It is easy to check that $\sigma_k^{-1}((0,\infty)) \subseteq B_k$,
 so we are done.
\end{proof}

\begin{prop}
Let $\pi : E \to B \times \R$ be a $D$-numerable principal $G$-bundle.
Define $p$ to be the pullback
\[
\xymatrix{E_1 \ar[r] \ar[d]_p & E \ar[d]^\pi \\ B \ar[r]_-i & B \times \R}
\]
where $i(b) = (b,1)$.
Then there exists an isomorphism of principal $G$-bundles:
\[
\xymatrix@C5pt{\pi^{-1}(B \times [0,1]) \ar[rr]^-\alpha \ar[dr]_{\pi|_{B \times [0,1]}} && E_1 \times [0,1] \ar[dl]^{p \times 1_{[0,1]}} \\ 
& B \times [0,1].}
\] 
\end{prop}

\begin{proof}
We first show that there is a commutative diagram in $\Diff$
\begin{equation}\label{eq:goal1}
\cxymatrix{\pi^{-1}(B \times [0,1]) \ar[d]_{\pi|_{B \times [0,1]}} \ar[r]^f & \pi^{-1}(B \times [0,1]) \ar[d]^{\pi|_{B \times [0,1]}} \\
          B \times [0,1] \ar[r]_r & B \times [0,1] ,}
\end{equation}
where $f$ is $G$-equivariant and $r(b,t) = (b,1)$.
By the previous proposition, there is a smooth partition of unity $\{\rho_k:B \to \R\}_{k \in K}$
subordinate to a $D$-open cover $\{B_k\}_{k \in K}$ of $B$ 
such that $\pi$ is trivial over $B_k \times [0,1]$ for each $k$.
As in the proof of Lemma~\ref{lem:Bourbaki},
define $\sigma : B \to \R$ by $\sigma(b) = \sum_k \rho_k(b)^2$.
Note that $\sigma$ is smooth, nowhere zero and
$\sigma(b) \leq \sup_k \rho_k(b)$.
Let $u_k(b) = \phi(\rho_k(b)/\sigma(b))$, where $\phi:\R \to \R$ is a smooth function
such that $\phi(t) = 0$ for $t \leq 0$, $\phi(t) = 1$ for $t \geq 1$ and $\im(\phi)=[0,1]$.
Then $\sup_k u_k(b) = 1$ for each $b$ and $\supp(u_k) \subseteq B_k$.

For each $k$, define $r_k : B \times [0,1] \to B \times [0,1]$ by
$r_k(b,t) = (b, H(u_k(b), t))$, where $H : [0,1] \times [0,1] \to [0,1]$ is
defined by $H(s,t) = (1-t) s + t$.
Note that if $u_k(b) = 0$, $r_k(b,t) = (b,t)$, so for any given $b$,
only finitely many $r_k$'s are not the identity.
Also, if $u_k(b) = 1$, then $r_k(b,t) = (b, 1)$.
Now choose a $G$-equivariant trivialization $h_k : B_k \times [0,1] \times G \to \pi^{-1}(B_k \times [0,1])$
and define a function $f_k : E \to E$ over $r_k$ by setting
$f_k(h_k(b,t,g)) = h_k(r_k(b,t), g)$ for $b$ in $B_k$ and $f_k(x) = x$ otherwise. Then $f_k$ is $G$-equivariant.
Since $r_k$ is the identity outside of the support of $u_k$, $f_k$ is smooth.

Choose a total ordering of the indexing set $K$.
Define $f : E \to E$ to be the composite $f_{k_n} \circ \cdots \circ f_{k_1}$ on $\pi^{-1}(\{b\} \times [0,1])$,
where $\{k_1, \ldots, k_n\} = \{ k \in K \mid u_k(b) \neq 0 \}$ and $k_1 < \cdots < k_n$.
This respects the $G$-action, and lies over $r_{k_n} \circ \cdots \circ r_{k_1}$.
The latter composite sends $(b,t)$ to $(b,1)$, since at least one $r_{k_i}$ does,
and every $r_k$ sends $(b,1)$ to $(b,1)$.

It remains to show that $f$ is smooth, and it suffices to check this on an open cover.
For each $b$ in $B$, choose a $D$-open neighbourhood $U$ of $b$
so that $\{ k \in K \mid U \cap B_k \neq \emptyset \}$ is finite, enumerated as
$\{j_1, \ldots, j_n\}$ with $j_1 < \cdots < j_n$.
Then, on $\pi^{-1}(U \times [0,1])$, we have that $f$ is equal to the composite
$f_{j_n} \cdots f_{j_1}$, 
since a map $f_j$ is the identity over $\{b\} \times [0,1]$ if $u_j(b) = 0$.
This shows that $f$ is locally smooth and therefore smooth.
Thus, we have the required diagram~\eqref{eq:goal1}.

Since $r$ factors through $i : B \to B \times [0,1]$ and $p$ is a pullback, we get
a commutative square
\[
  \xymatrix{\pi^{-1}(B \times [0,1]) \ar[d]_{\pi|_{B \times [0,1]}} \ar[r] & E_1 \ar[d]^p \\
          B \times [0,1] \ar[r]_-{p_1} & B ,}
\]
where $p_1$ is the projection.
By Proposition~\ref{prop:commsq=>pullbackdiff}, $\pi|_{B \times [0,1]}$ is isomorphic to the 
pullback of $p$ along $p_1$, which is the product $p \times 1_{[0,1]}$, as required.
\end{proof}

\begin{cor}\label{cor:homotopy-pullback}
If $\pi : E' \to B'$ is a $D$-numerable principal $G$-bundle,
and $f$ and $g$ are smoothly homotopic maps $B \to B'$,
then the pullbacks $f^*(\pi)$ and $g^*(\pi)$ are isomorphic as principal $G$-bundles over $B$.
\end{cor}

\begin{proof}
Let $F:B \times \R \to B'$ be a smooth homotopy between $f$ and $g$.
Then $F^*(\pi)$ is a $D$-numerable principal $G$-bundle over $B \times \R$ by 
Lemma~\ref{lem:pullback-D-numerable-G}.
By the previous proposition, $F^*(\pi)$ is isomorphic
to a product $E_1 \times [0,1] \to B \times [0,1]$ for a certain principal $G$-bundle
$p : E_1 \to B$.
Thus the restrictions to $B \times \{0\}$ and $B \times \{1\}$ are both isomorphic to $p$.
\end{proof}

Recall that we saw in Section~\ref{se:no-classifying} that this property does not 
hold for an arbitrary principal $G$-bundle.

\begin{cor}
If $\pi : E' \to B'$ is a $D$-numerable diffeological bundle,
and $f$ and $g$ are smoothly homotopic maps $B \to B'$,
then the pullbacks $f^*(\pi)$ and $g^*(\pi)$ are isomorphic as diffeological bundles over $B$.
\end{cor}

\begin{proof}
This follows from~\cite[8.16]{I2} (see Section~\ref{se:classify-bundle}) and
Corollary~\ref{cor:homotopy-pullback}.
\end{proof}

\section{Classifying $D$-numerable principal bundles}
\label{se:classify-Dpb}

In this section, which forms the heart of the paper,
we construct a classifying space for all $D$-numerable principal bundles.

Let $G$ be a diffeological group with identity $e$.  Consider the infinite simplex
\[
\Delta^\omega := \{(t_0,t_1,\ldots) \in \oplus_\omega \R \,\mid\, \sum_{i=0}^\infty t_i = 1 \text{ and } t_i \geq 0 \text{ for each } i\},
\]
equipped with the sub-diffeology of $\oplus_\omega \R$, 
where $\oplus_\omega \R$ is the coproduct of countably many 
copies of $\R$ in $\DVect$ (see~\cite[Proposition~3.2]{Wu}). 
Explicitly, a function $t : U \to \Delta^{\omega}$ is a plot if and only if each
component function $t_i : U \to \R$ is smooth 
and for each $u \in U$ there are an open neighbourhood $V$ of $u$ and $n \in \N$
such that $t_i(v) = 0$ for all $v \in V$ and $i > n$. 
Put another way, any such plot $t$ locally lands in 
\[
\Delta^n := \{(t_0,t_1,\ldots,t_n) \in \R^{n+1} \,\mid\, \sum_{i=0}^n t_i = 1 \text{ and } t_i \geq 0 \text{ for each } i\}
\]
for some $n$, where $\Delta^n$ has the sub-diffeology of $\R^{n+1}$, and is
also naturally a diffeological subspace of $\Delta^{\omega}$.

On $\Delta^\omega \times \prod_\omega G$, 
define $(t_i,g_i) \sim (t_i',g_i')$ if the following conditions are satisfied: 
\begin{enumerate}
 \item $t_i=t_i'$ for each $i \in \omega$;
 \item if $t_i = t_i' \neq 0$, then $g_i=g_i'$.
\end{enumerate} 
This is an equivalence relation on $\Delta^\omega \times \prod_\omega G$, and we write $EG$ for the 
quotient diffeological space and $[t_i, g_i]_{EG}$ or simply $[t_i, g_i]$ for an equivalence class.

Now we consider group actions. Define $(\Delta^\omega \times \prod_\omega G) \times G \to \Delta^\omega \times \prod_\omega G$
by $((t_i,g_i),g) \mapsto (t_i,g_i g)$. Note that this is smooth and compatible with the equivalence relation $\sim$, and hence induces a
smooth right action $EG \times G \to EG$. It is easy to see that this action is free, i.e., $[t_i,g_i] \cdot g = [t_i,g_i]$ 
implies that $g=e$. We write $BG$ for the corresponding orbit space with the quotient diffeology
and write elements in $BG$ as $[t_i,g_i]_{BG}$
or simply $[t_i,g_i]$ if no confusion will occur. 

Both $E$ and $B$ are functors from the 
category of diffeological groups and smooth group homomorphisms to $\Diff$. 

Our first goal is to show that the quotient map $\pi : EG \to BG$ is a $D$-numerable
principal $G$-bundle.  This requires a lemma that we will use implicitly in various
places, and a remark.

\begin{lem}
The function $f_i : BG \to \R$ sending $[t_j,g_j]$ to $t_i$ is smooth for each $i$. 
\end{lem}

\begin{proof}
It suffices to show that the composite $\Delta^{\omega} \times \prod_{\omega} G \to EG \to BG \to \R$ is smooth,
where the first two maps are the quotient maps and the third map is $f_i$.
This composite is equal to the composite
$\Delta^{\omega} \times \prod_{\omega} G \to \Delta^{\omega} \hookrightarrow \oplus_{\omega} \R \to \R$,
where the first map is the projection, the second is the inclusion, and the third is projection onto the
$i^{th}$ summand, all of which are smooth.
\end{proof}

\begin{rem}\label{rem:EG_n}
Any plot $p : U \to EG$ locally factors through the quotient map 
$\Delta^{\omega} \times \prod_{\omega} G \to EG$.
Therefore, by the description of the diffeology on $\Delta^{\omega}$, it
locally lands in $\Delta^n \times \prod_{\omega} G$ for some $n$.
This lift can be adjusted so that its values $(t_i, g_i)$ have
$g_i = e$ for $i > n$, which means that it factors through the natural map
from $\Delta^n \times G^{n+1}$.
In particular, if we let $EG_n \subseteq EG$ consist of those points $[t_i, g_i]$ with $t_i = 0$ for all $i > n$,
then $p$ locally factors through $EG_n$ for some $n$.

Similarly, a plot $p : U \to BG$ locally factors through $\Delta^n \times G^{n+1}$ for some $n$.
In particular, if we define $BG_n \subseteq BG$ analogously, $p$ locally factors through
$BG_n$ for some $n$.

Another way to phrase these facts is to say that $EG = \colim EG_n$ and $BG = \colim BG_n$,
where the colimits are in the category of diffeological spaces.
It follows from this and~\cite[Lemmas~3.17 and~4.1]{CSW} that if $D(G)$ is a locally compact
Hausdorff topological group, then $D(BG) \cong B_{\textrm{Top}}(D(G))$, where the right-hand-side denotes the
usual classifying space construction applied to the topological group $D(G)$.
\end{rem}

\begin{thm}\label{thm:EG->BG-D-numerable-principal}
 The quotient map $\pi:EG \to BG$ is a $D$-numerable principal $G$-bundle.
\end{thm}

\begin{proof}
 We first show that $\pi$ is a locally trivial principal $G$-bundle.
 Let $B_i := \{[t_j,g_j] \in BG \mid t_i > 1/2^{i+1}\}$. Then $B_i$ is $D$-open in $BG$. Since 
 $\sum_{i=0}^\infty 1/2^{i+1} = 1$, we see that the $B_i$'s cover $BG$.
 We claim that $\pi|_{B_i}$ is trivial for each $i$. 
 Define $h : B_i \times G \to \pi^{-1}(B_i)$ by sending $([t_j, g_j]_{BG}, g)$
 to $[t_j, g_j g_i^{-1} g]_{EG}$.
 This is smooth, $G$-equivariant, and commutes with the projections to $B_i$.
 Its inverse sends $[t_j, g_j]_{EG}$ in $\pi^{-1}(B_i)$ to $([t_j, g_j]_{BG}, g_i)$,
 and is therefore also smooth and $G$-equivariant.
 So $\pi$ is locally trivial.

 Now we show that $\pi$ is $D$-numerable.
 We first claim that the $D$-open cover $\{ B_i \}$ is locally finite.
 Fix $[\bar{t}_j, \bar{g}_j] \in BG$, and choose $N$ so that $\bar{t}_i=0$ for all $i>N$. 
 Let $B := \{[t_j,g_j] \in BG \mid t_i < 1/2^{i+1} \text{ for all } i>N\}$. Then $[\bar{t}_j,\bar{g}_j] \in B$ and $B$ only intersects 
 finitely many $B_i$'s. We are left to show that $B$ is $D$-open. Let $p:U \to BG$ be a plot.
 By Remark~\ref{rem:EG_n}, we can replace $U$ by a smaller open subset so
 that there exist $n \in \N$ and a smooth map $U \to \Delta^n \times G^{n+1}$ such that the following 
 diagram commutes:
 \[
  \xymatrix{U \ar[d] \ar[dr]^p \\ \Delta^n \times G^{n+1} \ar[r] & BG.}
 \]
 Since the preimage of $B$ in $\Delta^n \times G^{n+1}$ under the horizontal map in the above diagram 
 is $D$-open, as it is a finite intersection of $D$-open subsets, $p^{-1}(B)$ is open in $U$. Hence $B$ is $D$-open in $BG$.

 Fix any smooth function $\rho:\R \to \R$ such that $\rho(t)=0$ for all $t \leq 0$ and $\rho$ is strictly increasing on 
 $(0,\infty)$. Define $\tau_i:BG \to \R$ by 
 \[
  \tau_i([t_j,g_j]) = \rho(t_i - 1/2^{i+1}).
 \]
 Note that $B_i = \tau_i^{-1}((0, \infty))$, so these open supports are locally finite.
 Since $\sum_{i=0}^n 1/2^{i+1} < 1$ for each $n$, at least one $\tau_i$ is nonzero at
 each point, so we can normalize the $\tau_i$'s to obtain a smooth partition of unity.
 By Lemma~\ref{lem:Bourbaki}, we obtain a smooth partition of unity subordinate to the
 trivializing open cover $\{B_i\}$ of $BG$.
 Therefore, $\pi:EG \to BG$ is $D$-numerable. 
\end{proof}

The next result will imply that $EG$ is contractible and is a key
step in proving that $\pi : EG \to BG$ is a universal $D$-numerable bundle.

\begin{prop}\label{prop:EG-subterminal}
Let $E$ be any diffeological space with a right $G$-action,
and let $h_0, h_1 : E \to EG$ be $G$-equivariant maps.
Then there is a smooth $G$-equivariant homotopy $h_0 \simeq h_1$.
\end{prop}

By a $G$-equivariant homotopy, we mean a homotopy through $G$-equivariant maps.

\begin{proof}
Fix a smooth function $\rho:\R \to \R$ such that there exists $\epsilon >0$ with $\rho(t)=0$ if $t < \epsilon$, 
$\rho(t)=1$ if $t > 1 - \epsilon$, and $\im(\rho) = [0,1]$.
Define $H^{\od} :EG \times \R \to EG$ by sending $([t_i, g_i], t)$ to $[t_i', g_i']$
defined as follows.
If $t \leq 0$, then $[t_i', g_i'] = [t_i, g_i]$.
If $t$ is in the interval $\left[ \frac{1}{n+1}, \frac{1}{n} \right)$ for $n \in \N^{>0}$,
then
\begin{align*}
t_i' &= \begin{cases}
\hfill                        t_i,\hspace{12pt} & \text{if $i < n$,} \\
            (1 - \alpha(t))   t_{n+j},           & \text{if $i = n + 2j$ for $j \in \N$,} \\
\hfill           \alpha(t) \, t_{n+j},           & \text{if $i = n + 2j + 1$ for $j \in \N$,} \\
        \end{cases}\\
\intertext{where}
\alpha(t) &= \rho \! \left( \frac{t-\frac{1}{n+1}}{\frac{1}{n}-\frac{1}{n+1}} \right), \\
\intertext{and}
g_i' &= \begin{cases}
                              g_i,              & \text{if $i < n$,} \\
                              g_{n+j},           & \text{if $i = n + 2j$ for $j \in \N$,} \\
                              g_{n+j},           & \text{if $i = n + 2j + 1$ for $j \in \N$.} \\
        \end{cases}
\end{align*}
If $t \geq 1$, then $t_{2j}' = t_j$, $t_{2j+1}' = 0$, $g_{2j}' = g_j$ and $g_{2j+1}' = e$ for $j \in \N$.
Although $g_i$ is not well-defined when $t_i = 0$, $H^{\od}$ is well-defined.
Also, $H^{\od}|_{t=0} = 1_{EG}$ and $H^{\od}|_{t=1}$ lands in the
subset $EG^{\od} := \{ [t_i, g_i] \in EG \mid t_i = 0 \text{ for $i$ odd}\}$.
One can see that $H^{\od}$ is smooth, using that
every plot of $EG$ locally factors through $\Delta^n \times G^{n+1}$ for some $n$
(Remark~\ref{rem:EG_n}).
Also, $H^{\od}$ is a homotopy through $G$-equivariant maps.
It follows that $h_0$ is $G$-equivariantly homotopic to a map $h_0'$ landing in $EG^{\od}$.

Similarly, we can show that $h_1$ is $G$-equivariantly homotopic to a map $h_1'$ landing
in $EG^{\ev} := \{ [t_i, g_i] \in EG \mid t_i = 0 \text{ for $i$ even}\}$.

Now define $H : E \times \R \to EG$ as follows.  Given $(x,t) \in E \times \R$,
suppose $h_s'(x) = [t^s_i, g^s_i]$ for $s = 0, 1$.
Define $H(x,t)$ to be $[t_i, g_i]$, where
\begin{align*}
 t_i &= \begin{cases}
            (1 - \rho(t))   t^0_i, & \text{if $i$ is even,} \\
          \hfill \rho(t) \, t^1_i, & \text{if $i$ is odd,}  \\
        \end{cases}\\
 g_i &= \begin{cases}
            \hspace*{44pt}  g^0_i, & \text{if $i$ is even,} \\
            \hspace*{44pt}  g^1_i, & \text{if $i$ is odd.}  \\
        \end{cases}
\end{align*}
Although $g^s_i$ is not well-defined when $t^s_i = 0$, $H(x,t)$ is well-defined.
In fact, by Remark~\ref{rem:EG_n}, we can locally make smooth choices of 
representatives $g^s_i$, which shows that $H$ is smooth.
Since $h_0'$ and $h_1'$ are $G$-equivariant, so is $H$.
And $H$ is a homotopy between $h_0'$ and $h_1'$,
which shows that $h_0$ and $h_1$ are smoothly $G$-equivariantly homotopic.
\end{proof}

\begin{cor}\label{cor:EG-contractible}
For any diffeological group $G$, $EG$ is smoothly contractible.
\end{cor}

\begin{proof}
Let $B$ be any diffeological space.  Then smooth maps $B \to EG$ biject
with $G$-equivariant maps $B \times G \to EG$.
Given two smooth maps $f_0, f_1 : B \to EG$, the associated maps
$B \times G \to EG$ are smoothly homotopic, by Proposition~\ref{prop:EG-subterminal}.
Restricting to $e \in G$ gives a smooth homotopy $f_0 \simeq f_1$.
Therefore, $EG$ is smoothly contractible.
\end{proof}

\begin{rem}
Since every diffeological group is fibrant (\cite[Proposition~4.30]{CW1}), 
and every diffeological bundle with fibrant fiber is a fibration (\cite[Proposition~4.28]{CW1}), 
we know that $\pi:EG \to BG$ is always a fibration. 
Also, by the long exact sequence of smooth homotopy groups of a diffeological bundle (\cite[8.21]{I2}) together with 
Corollary~\ref{cor:EG-contractible}, we have a group isomorphism $\pi_{n+1}^D(BG,b) \cong \pi_n^D(G,e)$ for every 
$n \in \N$ and $b \in BG$.
In addition, $BG$ is path-connected.  Indeed, given a point $[t_i, g_i]$ in $BG$,
choose a path in the infinite simplex from $(t_i)$ to $(1, 0, 0, \ldots)$.
This gives a path in $BG$ from $[t_i, g_i]$ to $[(1, 0, 0, \ldots), (g_0, g_1, \ldots)] = [(1, 0, 0, \ldots), (e, e, \ldots)]$.
\end{rem}

\begin{de}
Let $G$ be a diffeological group and let $B$ be a diffeological space. 
Write $\Prin_G(B)$ (resp.\ $\Prin_G^D(B)$) for the set of all (resp.\ $D$-numerable) principal $G$-bundles over $B$ 
modulo isomorphism of principal $G$-bundles.
Let
\[
  \theta:[B,BG] \to \Prin_G^D(B)
\]
be defined by $[f] \mapsto f^*(\pi:EG \to BG)$.
This is well-defined by Corollary~\ref{cor:homotopy-pullback}.
\end{de}

The final goal of this section is to prove that $\theta$ is a bijection for every $B$.
We break the proof into two propositions.

\begin{prop}\label{prop:theta-injective}
The map $\theta:[B,BG] \to \Prin_G^D(B)$ is injective.
\end{prop}

\begin{proof}
Let $f_0,f_1:B \to BG$ be smooth maps such that $f_0^*(\pi:EG \to BG)$ and $f_1^*(\pi:EG \to BG)$ are isomorphic principal 
$G$-bundles over $B$. 
Say they are isomorphic to the principal $G$-bundle $p:E \to B$. 
Then there exist smooth maps $h_0,h_1:E \to EG$ making the following diagrams commutative:
\[
\begin{minipage}[b]{0.5\linewidth}
\xymatrix{E \ar[r]^-{h_0} \ar[d]_p & EG \ar[d]^\pi \\ B \ar[r]_-{f_0} & BG}
\end{minipage}
\hspace{2cm}
\begin{minipage}[b]{0.5\linewidth}
\xymatrix{E \ar[r]^-{h_1} \ar[d]_p & EG \ar[d]^\pi \\ B \ar[r]_-{f_1} & BG.}
\end{minipage}
\]
By Proposition~\ref{prop:EG-subterminal}, there is a smooth $G$-equivariant homotopy $H$
between $h_0$ and $h_1$.
By $G$-equivariance, $H$ induces a smooth homotopy $f_0 \simeq f_1$.
\end{proof} 

\begin{prop}\label{prop:theta-surjective}
The map $\theta:[B,BG] \to \Prin_G^D(B)$ is surjective.
\end{prop}

\begin{proof}
Let $p:E \to B$ be a $D$-numerable principal $G$-bundle. By Proposition~\ref{prop:commsq=>pullbackdiff}, it is enough 
to show that there exist a $G$-equivariant smooth map $f:E \to EG$ and a smooth map $g:B \to BG$ making the following diagram 
commutative:
\[
\xymatrix{E \ar[d]_p \ar[r]^f & EG \ar[d]^\pi \\ B \ar[r]_g & BG.}
\]
By Proposition~\ref{prop:countable}, there exists a countable smooth partition of unity $\{\tau_n:B \to \R\}_{n \in \N}$
subordinate to a locally finite $D$-open cover $\{B_n\}_{n \in \N}$ of $B$ such that 
$p:p^{-1}(B_n) \to B_n$ is trivial for each $n$. 
Let $h_n:B_n \times G \to p^{-1}(B_n)$ be a $G$-equivariant trivialization over $B_n$, 
and let $q_n:B_n \times G \to G$ be the projection. 
Define $f:E \to EG$ by $x \mapsto [\tau_i(p(x)),q_i(h_i^{-1}(x))]$. 
Note that whenever $h_i^{-1}(x)$ is undefined, $\tau_i(p(x)) = 0$, and we define $q_i(h_i^{-1}(x)) = e$.
Hence, $f$ is well-defined. It is easy to check that $f$ is $G$-equivariant, and therefore induces a function 
$g:B \to BG$ making the required square commutative. So we are left to show that $f$ is smooth. 
Since $\{p^{-1}(B_n)\}_{n \in \N}$ is a locally finite $D$-open cover of $E$, for every $x \in E$, 
there exists a $D$-open subset $V$ of $x$ in $E$ such that $V$ only intersects $p^{-1}(B_{i_1}),\ldots,p^{-1}(B_{i_s})$ 
for a finite subset $I_V := \{i_1,\ldots,i_s\} \subset \N$. Then $I_V$ is a disjoint union of $I_V'$ and $I_V''$ with 
$x \in p^{-1}(B_i)$ for every $i \in I_V'$ and $x \notin p^{-1}(B_j)$ for every $j \in I_V''$. 
Then $E_x := V \cap (\cap_{i \in I_V'} p^{-1}(B_i)) \cap (\cap_{j \in I_V''} E \setminus p^{-1}(\supp(\tau_j)))$ is a 
$D$-open neighborhood of $x$ in $E$. By definition of $f$ and $EG$, it is clear that $f|_{E_x}$ is smooth, 
and therefore $f$ is smooth.
\end{proof}

In summary, we have proved:

\begin{thm}\label{thm:classify-principal}
For any diffeological space $B$ and any diffeological group $G$, the map $\theta : [B,BG] \to \Prin_G^D(B)$ 
is a bijection which is natural in $B$.
\end{thm}

The naturality of $\theta$ with respect to $G$ will be explained in Theorem~\ref{thm:naturality-G} in the next section.

\begin{ex}
For any smoothly contractible diffeological space $B$, the only $D$-numerable principal bundle over $B$ is the trivial bundle. 
For example, this applies when $B$ is an indiscrete diffeological space or a diffeological vector space.
\end{ex}

As an immediate consequence of the above theorem, we have:

\begin{cor}
Classifying spaces are unique up to smooth homotopy, in the sense that if a diffeological space $X$ has the property that
there is a bijection $[B,X] \to \Prin_G^D(B)$ which is natural in $B$, then $X$ is smoothly homotopy equivalent to $BG$.
\end{cor}

Note that this corollary uses the fact that we classify certain bundles over \emph{all} diffeological spaces. 
We use this to calculate some examples of $BG$:

\begin{prop}\label{prop:dvs-mod-G}
Let $V$ be a diffeological vector space, and let $G$ be an additive subgroup. Assume that the principal bundle 
$V \to V/G$ is $D$-numerable. Then $BG$ is smoothly homotopy equivalent to $V/G$. In particular,
$B V$ is smoothly contractible and
$B \Z^n$ is smoothly homotopy equivalent to $T^n = (S^1)^n$.
\end{prop}

\begin{proof}
By the universality of $EG \to BG$ (Theorem~\ref{thm:classify-principal}), we get a $G$-equivariant smooth map $f:V \to EG$. 
On the other hand, we can define $g:EG \to V$ by sending $[t_i,g_i]$ to $\sum_i t_i g_i$. 
It is not hard to show that $g$ is well-defined, smooth, and $G$-equivariant. 
By Proposition~\ref{prop:EG-subterminal}, we know that $f \circ g$ is $G$-equivariantly smoothly homotopic to $1_{EG}$. 
Since every $G$-equivariant smooth map $h:V \to V$ is $G$-equivariantly smoothly homotopic to $1_V$ via the affine homotopy 
$F(v,t) := t h(v) + (1-t) v$, we know that $g \circ f$ is $G$-equivariantly smoothly homotopic to $1_V$. 
Therefore, $EG$ is $G$-equivariantly smoothly homotopy equivalent to $V$.
It follows that $BG$ is smoothly homotopy equivalent to $V/G$.

Taking $G = V$, we have that $V \to V/G = *$ is a $D$-numerable principal $V$-bundle
and so $B V$ is smoothly homotopy equivalent to a point.
To see that $B \Z^n$ is smoothly homotopy equivalent to $T^n$,
take $V = \R^n$ and observe that we have a $D$-numerable principal 
$\Z^n$-bundle $\R^n \to \R^n/\Z^n \cong T^n$. 
\end{proof}

\begin{rem}
More generally, consider any additive subgroup $G$ of a diffeological vector space $V$.
As described in the proof of Proposition~\ref{prop:dvs-mod-G}, there
is a smooth, $G$-equivariant map $g : EG \to V$.
The map $g$ induces a smooth map $h : BG \to V/G$, and
by Proposition~\ref{prop:commsq=>pullbackdiff}, we know that $h^*(p)$ is
isomorphic to $EG \to BG$ as principal $G$-bundles over $BG$, 
where $p:V \to V/G$ is the projection.
It follows that every $D$-numerable principal bundle is isomorphic
to a pullback of $p$.
However, $p$ itself may not be $D$-numerable (consider $\R \to \R/\Q$)
and homotopic maps $X \to V/G$ may not give isomorphic pullbacks.
\end{rem}

On $\oplus_\omega \R$, we have a smooth inner product defined by $\langle (x_i),(y_i) \rangle = \sum_i x_i y_i$. 
Let $S^\infty$ be the subspace of $\oplus_\omega \R$ consisting of the elements of norm $1$. 
The discrete multiplicative group $\Z/2 = \{\pm 1\}$ acts on $S^\infty$ by $(x_i) \cdot (-1) = (-x_i)$. 
Write $\R P^\infty$ for the orbit space. 
Identifying $\oplus_{\omega} \R$ with $\oplus_{\omega} \C$, $S^{\infty}$ can also be thought of
as the unit vectors in $\oplus_{\omega} \C$.
Therefore, the Lie group $S^1$ acts on $S^\infty$ by pointwise multiplication. 
Write $\C P^\infty$ for the orbit space.

\begin{prop}\
\begin{enumerate}
\item $B \Z/2$ is smoothly homotopy equivalent to $\R P^\infty$.
\item $B S^1$ is smoothly homotopy equivalent to $\C P^\infty$.
\end{enumerate}
\end{prop}

\begin{proof}
(1) We first show that the quotient map $p:S^\infty \to \R P^\infty$ is a $D$-numerable principal $\Z/2$-bundle. 
Let $U_j := \{[x_i] \in \R P^\infty \mid |x_j| > 1/(2j+2)\}$. Then $\{U_j\}_{j \in \omega}$ is a $D$-open cover 
of $\R P^\infty$. Define $\mu_j:\R P^\infty \to \R$ by 
\[
\mu_j([x_i]) = \begin{cases} \exp(\frac{-1}{|x_j| - 1/(2j+2)}), & \textrm{if $|x_j| > 1/(2j+2)$,} \\ 0, & \textrm{else.} \end{cases}
\] 
Then $\mu_j$ is smooth, and $\mu_j^{-1}((0,\infty)) = U_j$. 
By an argument similar to the proof of Theorem~\ref{thm:EG->BG-D-numerable-principal},
one can show that $\{U_j\}_{j \in \omega}$ is locally finite. 
By Lemma~\ref{lem:Bourbaki}, there is a smooth partition of unity subordinate to $\{U_j\}_{j \in \omega}$.
It is straightforward to check that $p|_{U_j}$ is trivial for each $j$. 
Therefore, $p:S^\infty \to \R P^\infty$ is a $D$-numerable principal $\Z/2$-bundle. 

By the universality of $E \Z/2 \to B \Z/2$, we have a $\Z/2$-equivariant smooth map $f:S^\infty \to E \Z/2$. 
Define $g:E \Z/2 \to S^\infty$ by 
\[
g([t_i,g_i]) = \Bigg( \frac{g_i t_i}{\sqrt{\sum_j t_j^2}} \Bigg).
\]
It is smooth and $\Z/2$-equivariant. By Proposition~\ref{prop:EG-subterminal}, we know that $f \circ g$ is
$\Z/2$-equivariantly smoothly homotopic to $1_{E(\Z/2)}$. 

Next we show that $1_{S^\infty}$ is $\Z/2$-equivariantly smoothly homotopic to both $i_{\ev}$ and $i_{\od}$. 
Here $i_{\ev}:S^\infty \to S^\infty$ sends $(x_i)$ to $(y_i)$ with $y_{2i} = x_i$ and $y_{2i+1}=0$, 
and similarly, $i_{\od}:S^\infty \to S^\infty$ sends $(x_i)$ to $(y_i)$ with $y_{2i+1} = x_i$ and $y_{2i}=0$. 
We show $1_{S^\infty} \simeq_{\Z/2} i_{\ev}$ below, and the other case is similar. 

Fix a smooth function $\rho:\R \to \R$ such that there exists $\epsilon > 0$ with $\rho(t) = 0$ if $t < \epsilon$, 
$\rho(t) = \pi/2$ if $t > 1 - \epsilon$, and $\im(\rho) = [0,\pi/2]$.
Define $H:S^\infty \times \R \to S^\infty$ by sending $((x_i),t)$ to $(y_i)$ given as follows.
When $t \leq 0$, $y_i = x_i$. When $t \in [1/(n+1),1/n)$, 
\[
y_i = \begin{cases} x_i, & \textrm{if $i < n$,} \\ \cos(\alpha_n(t)) \, x_{n+j}, & \textrm{if $i = n+2j$ for $j \in \N$,} \\ 
                    \sin(\alpha_n(t)) \, x_{n+j}, & \textrm{if $i = n+2j+1$ for $j \in \N$,} \end{cases}
\]
where
\[
\alpha_n(t) = \rho \left( \frac{t - \frac{1}{n+1}}{\frac{1}{n} - \frac{1}{n+1}} \right).
\]
When $t \geq 1$, $y_{2i} = x_i$ and $y_{2i+1} = 0$. Since $H$ is smooth and $\Z/2$-equivariant, 
we have $1_{S^\infty} \simeq_{\Z/2} i_{\ev}$. 

Given any $\Z/2$-equivariant smooth map $h:S^\infty \to S^\infty$, define $K:S^\infty \times \R \to S^\infty$ 
by sending $(x,t)$ to $\cos(\rho(t)) \, i_{\ev}(h(x)) + \sin(\rho(t)) \, i_{\od}(x)$. 
Since $K$ is smooth and $\Z/2$-equivariant, we have $i_{\ev} \circ h \simeq_{\Z/2} i_{\od}$. 
So we have $h \simeq_{\Z/2} i_{\ev} \circ h \simeq_{\Z/2} i_{\od} \simeq_{\Z/2} 1_{S^\infty}$. 
Hence, $g \circ f$ is $\Z/2$-equivariantly smoothly homotopic to $1_{S^\infty}$. 

Therefore, $B \Z/2$ is smoothly homotopy equivalent to $\R P^\infty$.

(2) This can be proved similarly, by considering the $D$-numerable principal $S^1$-bundle $S^\infty \to \C P^\infty$.
\end{proof}

We also have:

\begin{prop}
Let $G$ and $H$ be diffeological groups.
Then $B(G \times H)$ and $BG \times BH$ are smoothly homotopy equivalent.
\end{prop}

\begin{proof}
There is a natural $(G \times H)$-equivariant smooth map $g : E(G \times H) \to EG \times EH$
defined by sending $[t_i, (g_i, h_i)]$ to $([t_i, g_i], [t_i, h_i])$.

The $D$-topology of a product is not the same as the product of the $D$-topologies
in general.  Nevertheless, if $U$ is $D$-open in $BG$ and $V$ is $D$-open in $BH$,
then $U \times V$ is $D$-open in $BG \times BH$.
Moreover, if $\{ \sigma_i \}_{i \in I}$ and $\{ \tau_j \}_{j \in J}$ are
smooth partitions of unity for $BG$ and $BH$ respectively, then
$\{ \rho_{ij} \}_{(i,j) \in I \times J}$ is a partition of unity for $BG \times BH$,
where $\rho_{ij}(x,y) := \sigma_i(x) \tau_j(y)$.
It follows that $EG \times EH \to BG \times BH$ is a $D$-numerable $(G \times H)$-principal bundle.
Therefore, we have a $(G \times H)$-equivariant smooth map $f : EG \times EH \to E(G \times H)$.

By Proposition~\ref{prop:EG-subterminal}, we know that $f \circ g$ is
$(G \times H)$-equivariantly smoothly homotopic to $1_{E(G \times H)}$.

If $X$ has a $(G \times H)$-action, then a $(G \times H)$-equivariant map
$X \to EG \times EH$ is the same as a $G$-equivariant map $X \to EG$
and an $H$-equivariant map $X \to EH$. 
Therefore, using Proposition~\ref{prop:EG-subterminal} on each factor,
we conclude that any two $(G \times H)$-equivariant maps $X \to EG \times EH$
are $(G \times H)$-equivariantly smoothly homotopic to each other.
In particular, $g \circ f$ is $(G \times H)$-equivariantly smoothly homotopic
to $1_{EG \times EH}$.

The claim follows.
\end{proof}

\section{Classifying $D$-numerable diffeological bundles}
\label{se:classify-bundle}

For diffeological spaces $B$ and $F$, write $\Bun_F(B)$ (resp.\ $\Bun_F^D(B)$) for the set of isomorphism classes of 
all (resp.\ $D$-numerable) diffeological bundles over $B$ with fiber $F$. 
This is a functor of $B$ under pullback of bundles.

It was shown in~\cite[8.16]{I2} that given a principal $G$-bundle $r:E \to B$ and a diffeological 
space $F$ with a left $G$-action, we can form an associated diffeological bundle $t:E \times_G F \to B$ with fiber $F$.
Here $E \times_G F := (E \times F)/{\sim}$, where $(y,f) \sim (y \cdot g,\, g^{-1} \cdot f)$ for all $g \in G$, and $t([y,f]) = r(y)$. 
Moreover, if $r$ is trivial (as a principal $G$-bundle), then so is $t$ (as a diffeological bundle).
This gives a natural transformation $\assoc : \Prin_G(B) \to \Bun_F(B)$ that depends on $F$ and the $G$-action,
and sends $D$-numerable bundles to $D$-numerable bundles.

On the other hand, it was shown in~\cite[8.14]{I2} that given a diffeological bundle $\pi:E \to B$ with fiber $F$,
there exists an associated principal $\Diff(F)$-bundle $s:E' \to B$ which we call the \dfn{frame bundle}.
As a set, $E' = \coprod_{b \in B} \Diff(F_b,F)$, where $F_b = \pi^{-1}(b)$ and $\Diff(F_b,F)$ consists of all diffeomorphisms $F_b \to F$.
The map $s$ sends $f:F_b \to F$ to $b$.
We equip $E'$ with the following diffeology: $p:U \to E'$ is a plot if and only if
all of the following conditions hold:
\begin{enumerate}
\item $s \circ p:U \to B$ is smooth;
\item $\{(u,y) \in U \times E \mid s(p(u)) = \pi(y)\} \to F$ sending $(u,y)$ to $p(u)(y)$ is smooth;
\item $U \times F \to E$ sending $(u,x)$ to $(p(u))^{-1}(x)$ is smooth. 
\end{enumerate}
The action $E' \times \Diff(F) \to E'$ is given by $(f,g) \mapsto g^{-1} \circ f$. 
Moreover, by~\cite[8.16]{I2}, if $\pi$ is trivial (as a diffeological bundle), then so is $s$ (as a principal 
$\Diff(F)$-bundle).
This gives a natural transformation $\frme : \Bun_F(B) \to \Prin_{\Diff(F)}(B)$ that
sends $D$-numerable bundles to $D$-numerable bundles.

In~\cite[8.16]{I2} it is shown that $\assoc \circ \frme$ is the identity, where
to define $\assoc$ we use the natural action of $\Diff(F)$ on $F$.
That is, if we start with a diffeological bundle $\pi:E \to B$ with fiber $F$,
form the associated principal $\Diff(F)$-bundle, and then take the associated $F$-bundle,
we get a bundle isomorphic to $\pi$.

In fact, these operations are inverse to each other:

\begin{thm}\label{thm:bijection-bundlevsprincipal}
We have a natural isomorphism $\assoc : \Prin_{\Diff(F)}(B) \to \Bun_F(B)$
which restricts to a natural isomorphism $\assoc : \Prin_{\Diff(F)}^D(B) \to \Bun_F^D(B)$.
\end{thm}

\begin{proof}
We saw that $\assoc \circ \frme$ is the identity.
To check that $\frme \circ \assoc$ is the identity, we start with a principal $\Diff(F)$-bundle $r:E \to B$
and show it is isomorphic to the frame bundle $s: E' \to B$ of
the associated bundle $t : E \times_{\Diff(F)} F \to B$.
It is enough to construct a $\Diff(F)$-equivariant smooth map $\alpha:E \to E'$ such that 
$s \circ \alpha = r$.
For $y \in E$ we define $\alpha(y) : t^{-1}(r(y)) \to F$ by sending $[x,f]$ in $E \times_{\Diff(F)} F$
to $\theta(f)$, where $\theta$ is the unique element of $\Diff(F)$ such that $x = y \cdot \theta$.
Such a $\theta$ exists because $r(x) = r(y)$, and $\alpha(y)$ is well-defined because
$[x \cdot \phi,\, \phi^{-1}(f)]$ is sent to $(\theta \phi)(\phi^{-1}(f)) = \theta(f)$ as well.
It is then not hard to check that $\alpha(y)$ 
is a diffeomorphism for each $y \in E$, and that $\alpha$ is $\Diff(F)$-equivariant and smooth. 

The last claim follows from the fact that both $\assoc$ and $\frme$ preserve $D$-numerable bundles.
\end{proof}

Combining this result with Theorem~\ref{thm:classify-principal}, we get:

\begin{thm}\label{thm:classify-diff-bundles}
There is a bijection $[B,B\Diff(F)] \to \Bun_F^D(B)$ which is natural in $B$.
\end{thm}

Using the techniques from this section, we can also show that the bijection 
in Theorem~\ref{thm:classify-principal} is natural with respect to the diffeological group.

\begin{de}[Functoriality of $\Prin_G(B)$]
Let $h:G \to G'$ be a smooth homomorphism between diffeological groups. 
Define a left action of $G$ on $G'$ by $g \cdot g' := h(g) g'$. 
Given a principal $G$-bundle $E \to B$, we can form the associated diffeological
bundle $E \times_G G' \to B$ with fiber $G'$.
We can define a right action of $G'$ on $E \times_G G'$ by $[x, g'] \cdot \hat{g}' := [x, g' \hat{g}']$.
One can check that this is a principal $G'$-bundle, and that this defines
a function $h_* : \Prin_G(B) \to \Prin_{G'}(B)$ making $\Prin_G(B)$ into a functor of $G$.
Moreover, if $E \to B$ is $D$-numerable, then so is $E' \to B$, 
so we see that $\Prin_G^D(B)$ is also functorial in $G$.
\end{de}

\begin{thm}\label{thm:naturality-G}
The bijection $\theta : [B, BG] \to \Prin_G^D(B)$ from Theorem~\ref{thm:classify-principal}
is natural in $G$.
That is, for any smooth homomorphism $h:G \to G'$ between diffeological groups,
the following diagram commutes:
\[
\xymatrix{[B,BG] \ar[r]^-\theta \ar[d]_{Bh_*} & \Prin_G^D(B) \ar[d]^{h_*} \\ [B,BG'] \ar[r]_-\theta & \Prin_{G'}^D(B).}
\]
\end{thm}

\begin{proof}
We first consider the universal case, where $B = BG$ and we start with
the identity map $BG \to BG$.
Define a map $EG \times G' \to EG'$ by sending $([t_i, g_i], g')$
to $[t_i, h(g_i) g']$, and notice that this is well-defined on the
associated principal $G'$-bundle $EG \times_G G'$.
It is also $G'$-equivariant, and makes the square
\[
  \xymatrix{
    EG \times_G G' \ar[d] \ar[r] & EG' \ar[d] \\
    BG \ar[r]_{Bh}              & BG' 
  }
\]
commute.  Thus, by Proposition~\ref{prop:commsq=>pullbackdiff}, it is a pullback
square, as required.

Now, given a map $f : B \to BG$, we compute the pullback of $EG'$ along
the composite $B \to BG \to BG'$ as
\begin{align*}
  B \times_{BG'} EG' &\cong B \times_{BG} (BG \times_{BG'} EG') & \text{(by functoriality of pullback)}\\
                    &\cong B \times_{BG} (EG \times_G G') & \text{(by the previous paragraph)} \\
                    &\cong (B \times_{BG} EG) \times_G G' & \text{(by naturality of $\assoc$)\rlap{,}}
\end{align*}
which shows that the square commutes.
\end{proof}

\section{Classifying $D$-numerable vector bundles}
\label{se:classify-vb}

We first recall the following definition from~\cite{CW2}:

\begin{de}
Let $B$ be a diffeological space. A \dfn{diffeological vector space over $B$} is a diffeological space $E$, a smooth 
map $\pi:E \to B$ and a vector space structure on each of the fibers $\pi^{-1}(b)$ such that the addition 
$E \times_B E \to E$, the scalar multiplication $\R \times E \to E$ and the zero section $B \to E$ are all smooth. 
\end{de}

In the case when $B$ is a point, we recover the concept of diffeological vector space. 
More generally, for any $b \in B$, $\pi^{-1}(b)$ equipped with the sub-diffeology of $E$ is a diffeological vector space.

\begin{lem}
Let $\pi:E \to B$ be a diffeological vector space over $B$, and let $f:B' \to B$ be a smooth map. 
Then the pullback $f^*(\pi)$ is a diffeological vector space over $B'$.
\end{lem}

\begin{proof}
This is straightforward.
\end{proof}

\begin{de}
Let $V$ be a diffeological vector space. A diffeological vector space $\pi:E \to B$ over $B$ is called \dfn{trivial of fiber type $V$} if there exists 
a diffeomorphism $h:E \to B \times V$ over $B$, such that for every $b \in B$, the restriction $h|_b:\pi^{-1}(b) \to V$ is an isomorphism of 
diffeological vector spaces. 

A diffeological vector space $\pi:E \to B$ over $B$ is called \dfn{locally trivial of fiber type $V$} if there exists a $D$-open cover $\{B_i\}$ of $B$ 
such that each restriction $\pi|_{B_i}:\pi^{-1}(B_i) \to B_i$ is trivial of fiber type $V$.

A diffeological vector space $\pi:E \to B$ over $B$ is called a \dfn{vector bundle of fiber type $V$} if the pullback along every plot of $B$ is 
locally trivial of fiber type $V$.  
\end{de}

\begin{de}
Let $V$ be a diffeological vector space. A vector bundle $\pi:E \to B$ of fiber type $V$ is called \dfn{$D$-numerable} if 
there exists a smooth partition of unity subordinate to a $D$-open cover $\{B_i\}_{i \in I}$ 
of $B$ such that each $\pi|_{B_i}$ is trivial of fiber type $V$. 
\end{de}

Let $V$ be a diffeological vector space and let $\GL(V)$ be the set of all linear isomorphisms $V \to V$ equipped with the 
sub-diffeology of $\Diff(V)$. Then $\GL(V)$ is a diffeological group. Let $G$ be a diffeological group. A \dfn{(left) linear $G$-action on $V$} 
is a smooth group homomorphism $G \to \GL(V)$. Given a principal $G$-bundle $r:E \to B$ and a linear $G$-action on $V$, 
we have an associated diffeological bundle $t:E \times_G V \to B$. 

\begin{lem}
Under the above assumptions, $t:E \times_G V \to B$ is a vector bundle of fiber type $V$.
\end{lem}

\begin{proof}
We make $E \times_G V \to B$ into a diffeological vector space over $B$ using the following maps.
The addition map
\[
  (E \times_G V) \times_B (E \times_G V) \to E \times_G V
\]
sends $([x,v],[x',v'])$ to $[x,v + g \cdot v']$,
where $g \in G$ is chosen so that $x' = x \cdot g$, which is possible since $r(x)=r(x')$.
The scalar multiplication map
\[
  \R \times (E \times_G V) \to E \times_G V
\]
sends $(\alpha, [x, v])$ to $[x, \alpha v]$.
And the zero section
\[
  B \to E \times_G V
\]
sends $b$ to $[x, 0]$, where $x$ is any element of $\pi^{-1}(b)$.
It is straightforward to check that these maps are all smooth and make
$t:E \times_G V \to B$ into a diffeological vector space over $B$,
and that $t$ is a vector bundle.
\end{proof}

Write $\VB_V(B)$ (resp.\ $\VB_V^D(B)$) for the set of isomorphism classes 
of (resp.\ $D$-numerable) vector bundles over $B$. Therefore, we have a 
natural transformation $\assoc:\Prin_G(B) \to \VB_V(B)$ that depends on the diffeological 
vector space $V$ and the linear $G$-action, and sends $D$-numerable bundles to 
$D$-numerable bundles.

On the other hand, given a vector bundle $\pi:E \to B$ of fiber type $V$, 
let $E''=\coprod_{b \in B} \Isom(\pi^{-1}(b),V)$ be equipped with the sub-diffeology of $E'$ defined in 
Section~\ref{se:classify-bundle}, where $\Isom(\pi^{-1}(b),V)$ denotes the set of all isomorphisms 
$\pi^{-1}(b) \to V$ of diffeological vector spaces. So we have a composite of smooth maps
$E'' \hookrightarrow E' \to B$, denoted by $s$, which sends each $f : \pi^{-1}(b) \to V$ to $b$.

\begin{lem}
Under the above assumptions, $s:E'' \to B$ is a principal $\GL(V)$-bundle.
\end{lem}
\begin{proof} 
It is easy to see that there is a commutative square
\[
\xymatrix{E'' \times \GL(V) \ar[d] \ar[r]^-{a''} & E'' \times E'' \ar[d] \\ E' \times \Diff(V) \ar[r]_-{a'} & E' \times E',}
\]
where the vertical maps are inclusions and the horizontal ones are the action maps 
as in Theorem~\ref{thm:principal}. Since all the other maps in the square are inductions, 
so is $a''$. Therefore, we have a commutative triangle 
\[
\xymatrix@C5pt{& E'' \ar[dl]_q \ar[dr]^s \\ X \ar[rr] && B,}
\]
where $X$ is the orbit space of $E''$ under the $GL(V)$-action, the quotient map $q$ is a principal bundle,
and the horizontal map is a smooth bijection. 
We will show that this horizontal map is a diffeomorphism, and for this
it is enough to show that $s:E'' \to B$ is a subduction. Let $p:U \to B$ be an arbitrary plot. 
Since $\pi:E \to B$ is a vector bundle, without loss of generality, we may assume that there is a diffeomorphism $\alpha:U \times V \to 
\{(u,x) \in U \times E \mid p(u)=\pi(x)\}$ over $U$ such that for each $u \in U$, the restriction 
$\alpha_u:V \to \pi^{-1}(p(u))$ is an isomorphism of diffeological vector spaces. 
It is then easy to check that $\hat{\alpha}:U \to E''$ defined by $\hat{\alpha}(u) := \alpha_u^{-1}$ is smooth.
This gives a commutative triangle 
\[
\xymatrix{& E'' \ar[d]^s \\ U \ar[ur]^{\hat{\alpha}} \ar[r]_p & B,}
\]
which implies that $s$ is a subduction.
\end{proof}

Therefore, we have a natural transformation $\frme:\VB_V(B) \to \Prin_{\GL(V)}(B)$ that sends 
$D$-numerable bundles to $D$-numerable bundles.

\begin{thm}\label{thm:classify-vb}
We have a natural isomorphism $\assoc:\Prin_{\GL(V)}(B) \to \VB_V(B)$ which restricts to a natural isomorphism 
$\assoc:\Prin_{\GL(V)}^D(B) \to \VB_V^D(B)$.
\end{thm}
\begin{proof}
The proof that $\frme \circ \assoc$ is the identity is the same as that of Theorem~\ref{thm:bijection-bundlevsprincipal}. 
Now we show that $\assoc \circ \frme$ is the identity. Let $\pi:E \to B$ be a vector bundle with fiber $V$. 
We need to show that $E'' \times_{\GL(V)} V \to B$ and $\pi$ are isomorphic vector bundles over $B$. 
One can check that $\alpha:E'' \times V \to E$ defined by $\alpha(f,v) := f^{-1}(v)$ is smooth.
Therefore, $\alpha$ induces a smooth bijection $\bar{\alpha}$ making the triangle
\[
\xymatrix@C5pt{E'' \times_{\GL(V)} V \ar[rr]^-{\bar{\alpha}} \ar[dr] && E \ar[dl]^\pi \\ & B}
\]
commute. So we are left to show that $\alpha$ is a subduction. 
This follows from the argument used in the proof of the previous lemma.

The last claim follows from the fact that both $\assoc$ and $\frme$ preserve $D$-numerable bundles.
\end{proof}

Combining this result with Theorem~\ref{thm:classify-principal}, we get:

\begin{thm}
There is a bijection $[B,B\GL(V)] \to \VB_V^D(B)$ which is natural in $B$.
\end{thm}

\appendix
\section{A smooth function that detects zeros}
\label{se:zeros}

In the proof of Proposition~\ref{prop:Dpb-on-product}, we used the existence of a certain smooth map
motivated by the continuous function $\min : C([0,1], \, \R^{\geq 0}) \to \R^{\geq 0}$.
We will give our smooth replace\-ment in Proposition~\ref{prop:functional-testing-zero}.

We equip $\R^{\geq 0}$ with the sub-diffeology of $\R$, so
$C^\infty(\R,\R^{\geq 0})$ consists of the smooth, non-negative functions $\R \to \R$,
and $C^\infty(\R,\R^{\geq 0})$ has the sub-diffeology of $C^\infty(\R,\R)$. 

We first need some lemmas:

\begin{lem}
 For any $f \in C^\infty(\R, \R^{\geq 0})$ such that $f(x) > 0$ for $x \in [0, 1]$, we have 
 \[
  C \exp \left( -C \int_0^1 \frac{1}{f(x)}dx \right) \leq \min_{x \in [0,1]} f(x),
 \]
 for any positive $C \geq \max_{x \in [0,1]} |f'(x)|$.
\end{lem}

\begin{proof}
 Write $m$ for $\min_{x \in [0,1]} f(x)$, so  $m = f(x_0)$ for some $x_0 \in [0,1]$. 
 By the mean value theorem, we have 
 \[
  f(x) - m = f(x) - f(x_0) \leq C|x-x_0|
 \]
 for all $x \in [0, 1]$.
 Therefore, 
 \begin{align*}
    \int_0^1 \frac{1}{f(x)} dx & \geq \int_0^1 \frac{1}{m + C|x-x_0|} dx \\
                               & = \int_0^{x_0} \frac{1}{m+C(x_0-x)}dx + \int_{x_0}^1 \frac{1}{m+C(x-x_0)}dx \\
                               & = \frac{1}{C} \ln \frac{m^2 + Cm + C^2 x_0(1-x_0)}{m^2} \\
                               & \geq \frac{1}{C} \ln \frac{m+C}{m}.
 \end{align*}
 So, 
 \[
  m \geq C \exp(-C \int_0^1 \frac{1}{f(x)}dx), 
 \]
 as required.
\end{proof}

As a special case, we obtain:

\begin{lem}\label{lem:less-than-min}
 For $f \in C^\infty(\R^2,\R^{\geq 0})$ and $a < b$, there exists $C > 0$ such that
 \[
  C \exp(-C \int_0^1 \frac{1}{f(t,x)}dx) \leq \min_{x \in [0,1]} f(t,x)
 \]
 for all $t \in [a,b]$ such that $f(t,x) > 0$ for all $x \in [0, 1]$.
\end{lem}

\begin{proof}
 By the previous lemma, it suffices to choose
 $C \geq \max_{x \in [0,1],t \in [a,b]} \big|\frac{\partial f}{\partial x}(t,x) \big|$.
\end{proof}

The last inequality we need for our result is:

\begin{lem}\label{lem:greater-than-c/t}
 Let $f \in C^\infty(\R^2,\R^{\geq 0})$ and $a < b$, and assume that
 $f(t_0, x_0) = 0$ for some $t_0 \in [a, b]$ and $x_0 \in [0, 1]$.
 Then there exists $c > 0$ such that
 \[
   \int_0^1 \frac{1}{f(t,x)} \, dx \geq \frac{c}{|t - t_0|}
 \]
 for all $t \in [a, b]$ such that $f(t,x) > 0$ for all $x \in [0,1]$.
\end{lem}

\begin{proof}
By translating $t$, we may assume that $t_0 = 0$, so $f(0, x_0) = 0$.
By enlarging the interval $[a, b]$, we may assume that it is symmetric about $0$.
By scaling $t$, we may assume that $[a, b] = [-1, 1]$.
Then, by smoothness, there exists $b > 0$ such that 
\[
  f(t,x) \leq b (t^2 + (x - x_0)^2)
\]
for every $t \in [-1,1]$ and $x \in [0,1]$. The squares comes from the fact that $f$ is assumed to be a non-negative smooth function, so $\partial_x f(0, x_0) =  \partial_t f(0, x_0) = 0$.
In particular, if $|x - x_0| \leq |t|$, then
\[
  f(t,x) \leq 2 b t^2 .
\]
Choose $t \in [-1,1]$ such that $f(t,x) > 0$ for all $x \in [0,1]$.
Integrating, we get
\[
  \int_0^1 \frac{1}{f(t,x)} \, dx
  \geq \int_{\max(x_0 - |t|,0)}^{\min(x_0 + |t|,1)} \frac{1}{f(t,x)} \, dx
  \geq \int_{\max(x_0 - |t|,0)}^{\min(x_0 + |t|,1)} \frac{1}{2 b t^2} \, dx
  \geq \frac{|t|}{2 b t^2} = \frac{c}{|t|},
\]
as claimed.  The last inequality comes from the fact that the interval of integration has width at least $|t|$.
\end{proof}

Now we prove the main result of this appendix.

\begin{prop}\label{prop:functional-testing-zero}
 There exists a smooth map $F:C^\infty(\R,\R^{\geq 0}) \to \R^{\geq 0}$ such that 
 $F(f)=0$ if and only if $f(x)=0$ for some $x \in [0,1]$.
\end{prop} 

\begin{proof}
 We are going to show that $F:C^\infty(\R,\R^{\geq 0}) \to \R^{\geq 0}$ defined by 
 \begin{equation}\label{eq:smoothmap}
  F(f)=\begin{cases} \exp(-\exp(\int_0^1 \frac{1}{f(x)} \, dx)), & \textrm{if $f(x)>0$ for all $x \in [0,1]$,} \\
                     0,                                      & \textrm{otherwise}
       \end{cases}
 \end{equation}
 satisfies the requirements.
 By definition of $F$, we have $F(f)=0$ if and only if $f(x)=0$ for some $x \in [0,1]$.
 We are left 
 to show the smoothness of $F$. By Boman's theorem (see, e.g., \cite[Corollary~3.14]{KM}), it is enough 
 to show that for every plot $p:\R \to C^\infty(\R,\R^{\geq 0})$, $F \circ p:\R \to \R$ is smooth.
 The map $\tilde{p} : \R \times \R \to \R$ defined by $\tilde{p}(t, x) = p(t)(x)$ is smooth and non-negative,
 and 
 \[
    (F \circ p)(t) =
            \begin{cases} \exp(-\exp(\int_0^1 \frac{1}{\tilde{p}(t,x)} \, dx)), & \textrm{if $\tilde{p}(t,x) > 0$ for all 
                                                                                                        $x \in [0,1]$,} \\
                          0,                                                & \textrm{otherwise}.
            \end{cases}
 \]
 
 It is easy to see that $A := \{t \in \R \mid \tilde{p}(t,x) > 0 \text{ for all } x \in [0,1]\}$ is open in $\R$, 
 and $F \circ p$ is smooth on $A$.
 To prove that $F \circ p$ is smooth, it suffices to show that for each $n \geq 0$
 and each $t_0 \in \R \setminus A$, $(F \circ p)^{(n)}(t_0)$ exists and is zero.
 We will prove this by induction on $n$.
 For $n = 0$, this holds by definition.

 For the inductive step, let $t_0 \in \R \setminus A$ and consider
 $(F \circ p)^{(n+1)}(t_0) = \lim_{t \to t_0} \frac{(F \circ p)^{(n)}(t)}{t-t_0}$.
 For $t \in \R \setminus A$, the numerator is zero by the inductive hypothesis.
 So we must bound
 \[
   \frac{|(F \circ p)^{(n)}(t)|}{|t-t_0|}
 \]
 for $t$ in $A$.

 By a separate induction, one can show that for $t \in A$
 \[
  (F \circ p)^{(n)}(t) = \exp(-G(t)) \, \sum_i G(t)^{n_i} \, \prod_j D_{ij}(t) ,
 \] 
 where $i$ ranges over a finite set, $G(t) = \exp(\int_0^1 \frac{1}{\tilde{p}(t,x)} \, dx)$, 
 each $n_i$ is in $\N$,
 $j$ ranges over a finite set depending on $i$,
 and $D_{ij}(t) = \int_0^1 \frac{E_{ij}(t,x)\ }{\tilde{p}(t,x)^{m_{ij}}} \, dx$
 with $E_{ij}$ a polynomial of finitely many iterated partial 
 derivatives of $\tilde{p}$ with respect to $t$ and $m_{ij} \in \N^{\geq 1}$.

 We fix an $i$ and a neighbourhood $[a, b]$ of $t_0$, and
 will bound $\exp(-G(t)) \, G(t)^{n_i} \, \prod_j D_{ij}(t)$ for $t \in [a, b] \cap A$.
 Using Lemma~\ref{lem:less-than-min}, choose $C>0$ so that
 $C G(t)^{-C} \leq m(t)$ for $t \in [a, b] \cap A$,
 where $m(t) := \min_{x \in [0,1]} \tilde{p}(t,x)$.
 For each $j$, let $M_{ij} := \max_{t \in [a,b], x \in [0,1]} \{|E_{ij}(t,x)|\}$
 and let $M_i := \prod_j M_{ij}$.
 Then for $t \in [a, b] \cap A$ we have 
 \begin{align*}
  \big| \exp(-G(t)) \, G(t)^{n_i} \, \prod_j D_{ij}(t) \big|
                          & = \exp(-G(t)) \, G(t)^{n_i} \,
                                              \prod_j \left| \int_0^1 \frac{E_{ij}(t,x)}{\tilde{p}(t,x)^{m_{ij}}} \, dx \right| \\
                          & \leq M_i \, \exp(-G(t)) \, G(t)^{n_i} \, 
                                                         \prod_j \int_0^1 \frac{1}{\tilde{p}(t,x)^{m_{ij}}} \, dx \\
                          & \leq M_i \, \exp(-G(t)) \, G(t)^{n_i} \, \prod_j \frac{1}{m(t)^{m_{ij}}} \\
                          & \leq M_i \, \exp(-G(t)) \, G(t)^{n_i} \, \prod_j \left(\frac{G(t)^C}{C}\right)^{m_{ij}} \\ 
                          & = M'_i \, \exp(-G(t)) \, G(t)^{n_i'}.
 \end{align*}
 Here, $M'_i = M_i \prod_j \frac{1}{C^{m_{ij}}}$ and $n_i' = n_i + \sum_j Cm_{ij}$ are constants. 

 By Lemma~\ref{lem:greater-than-c/t}, there is a $c > 0$ such that
 $\ln G(t) \geq c/|t - t_0|$ for $t$ in $[a, b] \cap A$.
 Therefore,
\[
  \frac{M'_i \, \exp(-G(t)) \, G(t)^{n_i'}}{|t - t_0|}
\leq \frac{M'_i}{c} \, \exp(-G(t)) \, G(t)^{n_i'} \, \ln G(t).
\]
 It also follows that $G(t) \to +\infty$ as $t \to t_0$ through $A$, and 
 so the right-hand-side goes to $0$ as $t \to t_0$ through $A$.

 Since this is true for each $i$ and by the induction hypothesis $(F \circ p)^{(n)}(t) = 0$ for $t \in \R \setminus A$,
 it follows that
\[
  (F \circ p)^{(n+1)} = \lim_{t \to t_0} \frac{(F \circ p)^{(n)}(t)}{t-t_0} = 0 .
\]
 This completes the inductive step and the proof of the proposition.
\end{proof}

We would like to thank Chengjie Yu for coming up with the formula~\eqref{eq:smoothmap}
and sketching the proof of the proposition, and Gord Sinnamon and Willie Wong for ideas that led 
towards this result.

\vspace*{10pt}
\end{document}